\documentclass[12pt,leqno]{amsart}
\usepackage[latin9]{inputenc}
\usepackage{verbatim}
\usepackage{amsthm}
\usepackage{amstext}
\usepackage{amssymb}
\usepackage{amscd}
\usepackage{url}

\makeatletter
\numberwithin{equation}{section}
\numberwithin{figure}{section}

\usepackage{amsthm}\usepackage{amscd}\usepackage{bm}
\theoremstyle{plain}
\newtheorem{theorem}{Theorem}[section]
\newtheorem{lemma}[theorem]{Lemma}

\newtheorem{proposition}[theorem]{Proposition}

\theoremstyle{definition}
\newtheorem{definition}[theorem]{Definition}
\newtheorem{remark}[theorem]{Remark}
\newtheorem{example}[theorem]{Example}

\newtheorem{problem}[theorem]{Problem}
\newtheorem{condition}[theorem]{Condition}
\newcommand{\affine}{\mathbb{C}}
\newcommand{\const}{\mathrm{const}}
\newcommand{\norm}[1]{\left|\!\left|#1\right|\!\right|}

\newcommand{\diam}{\mathrm{Diam}}
\newcommand{\moduli}{\mathcal{M}}
\newcommand{\widim}{\mathrm{Widim}}
\newcommand{\dist}{\mathrm{dist}}
\newcommand{\area}{\mathrm{Area}}
\newcommand{\nnorm}[1]{\left|\!\left|\!\left|#1\right|\!\right|\!\right|}
\newcommand{\mmdim}{\mathrm{dim}_{\mathrm{M}}}

\begin{document}

\title[Mean dimension of the dynamical system of Brody curves]
{Mean dimension of the dynamical system of Brody curves} 

\author[M. Tsukamoto]{Masaki Tsukamoto}



\thanks{This paper was supported by Grant-in-Aid for Young Scientists (B) 
25870334 from JSPS}

\date{\today}

\maketitle

\begin{abstract}
Mean dimension measures the size of an infinite dimensional dynamical system.
Brody curves are one-Lipschitz entire holomorphic curves in the projective space, and 
they form a topological dynamical system.
Gromov started the problem of estimating its mean dimension in the paper of 1999.
We solve this problem.
Namely we prove the exact mean dimension formula of the 
dynamical system of Brody curves.
Our formula expresses the mean dimension by the energy density of Brody curves.
The proof is based on a novel application of the metric mean dimension theory of 
Lindenstrauss and Weiss.
\end{abstract}

\section{Mean dimension formula} \label{section: mean dimension formula}

Let $z=x+\sqrt{-1}y$ be the standard coordinate in $\affine$.
Let $f=[f_0:f_1:\cdots:f_N]:\affine \to \affine P^N$ be a holomorphic curve
($f_i$: holomorphic function).
We define the \textbf{spherical derivative}
$|df|(z)\geq 0$ by 
\[ |df|^2(z) := \frac{1}{4\pi}\Delta\log(|f_0|^2+|f_1|^2+\dots+|f_N|^2) \quad 
   \left(\Delta := \frac{\partial^2}{\partial x^2}+\frac{\partial^2}{\partial y^2}\right).\]
This is the local Lipschitz constant. 
For a unit tangent vector 
$u\in T_z\affine$ the Fubini--Study length of $df(u)\in T_{f(z)}\affine P^N$ is equal to the spherical derivative 
$|df|(z)$.
A holomorphic curve $f:\affine\to \affine P^N$ is called a Brody curve if it is $1$-Lipschitz, namely $|df|\leq 1$
all over the plane.
The name comes from the work of Brody \cite{Brody}.
He proved that a projective variety is Kobayashi hyperbolic if and only if 
it does not contain a non-constant Brody curve.
(See Duval \cite{Duval} for a much deeper version of this 
\textit{Bloch--Brody principle}.)
The original motivation for the study of Brody curves comes from 
this connection to the hyperbolicity.

Recently new waves have begun, and 
Brody curves attract new interests of several authors
\cite{Barrett--Eremenko, Costa, Costa--Duval, Duval, Eremenko, 
Matsuo--Tsukamoto Brody curves, Tsukamoto moduli space of Brody curves, Tsukamoto packing, 
Winkelmann}.
One source of new interests is the work of Gromov \cite{Gromov}.
He introduced a new topological invariant of dynamical systems called \textbf{mean dimension},
and proposed a program to study many infinite dimensional dynamical systems in geometric analysis
from the viewpoint of this invariant.
\textit{The dynamical system consisting of Brody curves} is the simplest example in this program.
The purpose of our paper is to show that this perspective reveals a
new fundamental structure in holomorphic curve theory.

We define $\moduli(\affine P^N)$ as the space of all Brody curves $f:\affine \to \affine P^N$
endowed with the compact-open topology.
This is compact and the group $\affine$ continuously acts on it by 
\[ \affine \times \moduli(\affine P^N)\to \moduli(\affine P^N), \quad 
   (a,f(z))\mapsto f(z+a).\]
Our main object is the dynamics of this action.
The space $\moduli(\affine P^N)$ is infinite dimensional, 
and the topological entropy of the $\affine$-action is also infinite.
So $\moduli(\affine P^N)$ is very large.
Mean dimension is precisely an invariant introduced for this kind of large dynamical systems.
Its definition is reviewed in Subsection \ref{subsection: mean dimension}.
Here we explain intuition about this concept.
A basic example is the shift action of $\mathbb{Z}$ on the Hilbert cube $[0,1]^{\mathbb{Z}}$.
This is infinite dimensional and of infinite topological entropy.
Its mean dimension is one. Roughly speaking, this means that the system $[0,1]^\mathbb{Z}$
has one parameter per every unit second.
In general mean dimension counts the number of parameters ``averaged by a group action''.

The significance of mean dimension in topological dynamics was clarified by the works of 
Lindenstrauss and Weiss \cite{Lindenstrauss--Weiss, Lindenstrauss}.
Here is a sample of their results \cite[Theorem 5.1]{Lindenstrauss}.
Let $X$ be a compact metric space with a continuous $\mathbb{Z}$-action.
If $X$ can be equivariantly embedded into the system $[0,1]^\mathbb{Z}$ then its mean dimension 
is less than or equal to one.
Conversely, if $X$ is minimal (i.e. every orbit is dense) and its mean dimension is less than $1/36$, then it can be equivariantly 
embedded into $[0,1]^{\mathbb{Z}}$.
This striking embedding theorem shows that mean dimension properly measures the size of an 
infinite dimensional dynamical system.

We denote by $\dim(\moduli(\affine P^N):\affine)$ the mean dimension of the $\affine$-action on the space 
of Brody curves $\moduli(\affine P^N)$.
This is a nonnegative real number which 
counts the number of parameters in $\moduli(\affine P^N)$ ``per every unit area of the plane $\affine$''.
\begin{problem}[\textbf{Main problem}] \label{proble: main problem}
Compute the mean dimension $\dim(\moduli(\affine P^N):\affine)$.
\end{problem}
This problem was started by Gromov himself \cite[p. 396 (c)]{Gromov}.
He asked what kind of properties are reflected on
the estimate of the mean dimension.
In our notation, his result in \cite[p. 396 (c)]{Gromov} is the upper bound
\begin{equation} \label{eq: Gromov's upper bound}
   \dim(\moduli(\affine P^N):\affine) \leq 4N.
\end{equation}
The proof is based on the Nevanlinna theory 
(Eremenko \cite[Theorem 2.5]{Eremenko}).
In a series of works \cite{Tsukamoto moduli space of Brody curves, Tsukamoto deformation, Matsuo--Tsukamoto Brody curves},
we have been sharpening Gromov's estimate.
In the present paper we finally reach the answer to the problem.

Let $f:\affine \to \affine P^N$ be a Brody curve. We define its \textbf{energy density} $\rho(f)$ by 
\begin{equation} \label{eq: energy density}
 \rho(f) := \lim_{R\to \infty} \left(\frac{1}{\pi R^2} \sup_{a\in \affine} \int_{|z-a|<R} |df|^2 dxdy\right).
\end{equation}
This limit always exists (see Section \ref{subsection: energy density}).
It evaluates how densely the energy is distributed over the plane.
We define $\rho(\affine P^N)$ as the supremum of $\rho(f)$ over $f\in \moduli(\affine P^N)$.
It is known that (\cite{Tsukamoto packing})
\[ 0<\rho(\affine P^N) < 1, \quad \lim_{N\to \infty}\rho(\affine P^N) =1.\]
The main theorem is the formula expressing the mean dimension by the energy density. 

\begin{theorem}[\textbf{Main theorem}]\label{thm: main theorem}
\[ \dim(\moduli(\affine P^N):\affine) = 2(N+1)\rho(\affine P^N).\]
\end{theorem}

We briefly review preceding researches. 
Gromov's upper bound (\ref{eq: Gromov's upper bound}) is the first result on $\dim(\moduli(\affine P^N):\affine)$.
In \cite{Tsukamoto moduli space of Brody curves} we proved an improved upper bound
\begin{equation}\label{eq: upper bound by Nevanlinna}
  \dim(\moduli(\affine P^N):\affine) \leq 4N\rho(\affine P^N).
\end{equation}
The proof again used the Nevanlinna theory.
The lower bound on the mean dimension was developed 
in \cite{Tsukamoto deformation, Matsuo--Tsukamoto Brody curves}.
The main theorem of \cite{Matsuo--Tsukamoto Brody curves} is the sharp lower bound
\begin{equation} \label{eq: sharp lower bound}
   \dim(\moduli(\affine P^N):\affine) \geq 2(N+1)\rho(\affine P^N).
\end{equation}
Luckily this lower bound coincides with the upper bound (\ref{eq: upper bound by Nevanlinna})
when $N=1$. So we got the formula (\cite[Corollary 1.2]{Matsuo--Tsukamoto Brody curves})
\[ \dim(\moduli(\affine P^1):\affine) = 4\rho(\affine P^1).\]
This was the first nontrivial calculation of the mean dimension in geometric analysis.
When $N\geq 2$, the upper bound (\ref{eq: upper bound by Nevanlinna}) and lower bound 
(\ref{eq: sharp lower bound}) still have a gap.
The purpose of the present paper is to fulfill this gap by proving the sharp upper bound 
\begin{equation} \label{eq: sharp upper bound}
   \dim(\moduli(\affine P^N):\affine) \leq 2(N+1)\rho(\affine P^N).
\end{equation}
Combined with the lower bound (\ref{eq: sharp lower bound}), this establishes Theorem \ref{thm: main theorem}.

The proof of the sharp upper bound (\ref{eq: sharp upper bound}) requires an 
idea completely different from Gromov's upper bound (\ref{eq: Gromov's upper bound}) and our previous bound
(\ref{eq: upper bound by Nevanlinna}).
Before explaining a new idea, we review a previous argument.
Most known upper bounds on mean dimension are based on the techniques of 
\textbf{sampling and embedding} (Gromov \cite[Chapters 3 and 4]{Gromov}).
The proof of (\ref{eq: upper bound by Nevanlinna}) is a typical one, and it goes as follows.
Take a lattice $\Gamma\subset \affine$ satisfying 
$\area(\affine/\Gamma) < 1/(2\rho(\affine P^N))$.
By using Nevanlinna's first main theorem, we can prove that the \textit{sampling map} 
\[ \moduli(\affine P^N)\to (\affine P^N)^{\Gamma}, \quad f\mapsto f|_{\Gamma} = (f(\gamma))_{\gamma\in \Gamma},\]
is an embedding.
Then we get the upper bound on the mean dimension with respect to the $\Gamma$-action
\[ \dim(\moduli(\affine P^N):\Gamma) \leq \dim((\affine P^N)^{\Gamma}:\Gamma) = 2N.\]
The mean dimension of the $\affine$-action follows from this by 
\[ \dim(\moduli(\affine P^N):\affine) = \dim(\moduli(\affine P^N):\Gamma)/\area(\affine/\Gamma) \leq 2N/\area(\affine/\Gamma).\]
Letting $\area(\affine/\Gamma)\to 1/(2\rho(\affine P^N))$, we get (\ref{eq: upper bound by Nevanlinna}).
This proof is very simple but, unfortunately, non-flexible.
It is hopeless to prove the optimal bound (\ref{eq: sharp upper bound}) by this method,
and there also exist several problems facing similar difficulties.
For example, Gromov \cite[Chapter 4]{Gromov} studied a dynamical system consisting of complex subvarieties 
in $\affine^N$, and he proved an upper bound on the mean dimension by using the sampling argument.
But his estimate is very far from the conjectural value
\cite[p. 408 Corollary, p. 409 Remark, p. 409 Remarks and open questions (a)]{Gromov}.
Sampling method seems to be inadequate for obtaining a precise estimate.

In order to overcome these difficult situations, 
we started to develop a completely new approach in \cite{Tsukamoto Yang--Mills dynamics}.
The paper \cite{Tsukamoto Yang--Mills dynamics} examined a new technique
in the context of Yang--Mills gauge theory.
Based on this experience, now we attack to Brody curves.
A key novel ingredient of our approach is \textbf{metric mean dimension}
introduced by Lindenstrauss--Weiss \cite{Lindenstrauss--Weiss}.
This is a geometric/information theoretic version of mean dimension.
Its idea is as follows.
Given a dynamical system, suppose we try to store on computer the orbits of the system
over a very long period of time to an accuracy of $\varepsilon>0$ .
How many memory (bits) do we need?
Asymptotically (as the period of time goes to infinity and $\varepsilon$ goes to zero) the answer is given by metric mean dimension.
The precise definition is given in Subsection \ref{subsection: mean dimension}.

A fundamental theorem of Lindenstrauss--Weiss asserts that metric mean dimension is always an upper bound 
on mean dimension.
So we can approach to mean dimension via metric mean dimension.
This provides us great flexibility because 
\textit{metric mean dimension is a more local quantity than mean dimension}.
Intuitively speaking, when we store the information of a dynamical system on computer,
we can decompose the system into small pieces and try to memorize each piece separately.
Namely we can decompose a global problem into local ones which are more suitable for detailed analysis.
This enables us to apply the analytic machinery developed in \cite{Tsukamoto deformation, Matsuo--Tsukamoto Brody curves}
to the problem of proving the upper bound (\ref{eq: sharp upper bound}).
The techniques in \cite{Tsukamoto deformation, Matsuo--Tsukamoto Brody curves} were originally introduced for 
the lower bound (\ref{eq: sharp lower bound}).
Its (unexpected) usefulness for the upper bound seems to suggest that our method is a right way to the problem.
It is very likely that our approach can be also applied to other problems, and
we hope that it will grow to be a standard technique in mean dimension theory.

\section{Preliminaries}

\subsection{Mean dimension} \label{subsection: mean dimension}

We review basic definitions of mean dimension here.
The main references are Gromov \cite{Gromov} and Lindenstrauss--Weiss \cite{Lindenstrauss--Weiss}.
Readers can find further information in 
Lindenstrauss \cite{Lindenstrauss} and Gutman \cite{Gutman}.

First we need to introduce some metric invariants.
Let $(X,d)$ be a compact metric space.
Let $Y$ be a topological space, and $f:X\to Y$ a continuous map.
For a positive number $\varepsilon$, we call $f$ an \textbf{$\varepsilon$-embedding} if $\diam (f^{-1}(y))<\varepsilon$ 
for all $y\in Y$. 
This means that $f$ looks like an embedding if we ignore an error smaller than $\varepsilon$.
We define the \textbf{$\varepsilon$-width dimension} 
$\widim_\varepsilon(X,d)$ as the minimum integer $n\geq 0$ such that there exist an 
$n$-dimensional finite simplicial complex $P$ and an $\varepsilon$-embedding from $(X,d)$ to $P$.
The topological dimension is given by 
\[ \dim X = \lim_{\varepsilon\to 0} \widim_\varepsilon (X,d).\]

For $\varepsilon>0$ we set 
\begin{equation*}
   \begin{split}
   &\#(X,d,\varepsilon) := \min\{\, |\alpha|\, |\, \text{$\alpha$ is an open covering of $X$ with 
   $\diam U<\varepsilon$ for all $U\in \alpha$}\}, \\
   &\#_{\mathrm{sep}}(X,d,\varepsilon) := \max\{n\geq 1|\, \exists x_1,\dots,x_n\in X \text{ with }
   \dist(x_i,x_j)>\varepsilon \> (i\neq j)\}.
   \end{split}
\end{equation*} 
Here ``sep'' is the abbreviation for \textbf{separated set}.
These two quantities are essentially equivalent to each other:
For $0<\delta<\varepsilon/2$ 
\begin{equation} \label{eq: separated set and spanning set}
  \#_{\mathrm{sep}}(X,d,\varepsilon) \leq \#(X,d,\varepsilon) 
   \leq \#_{\mathrm{sep}}(X,d,\delta).
\end{equation}
The next lemma will be used later. Its proof is trivial
\begin{lemma} \label{lemma: separated set}
Let $(X,d)$ and $(Y,d')$ be compact metric spaces.
Let $\varepsilon, \delta>0$.
Suppose there exists a map (not necessarily continuous) $f:X\to Y$ satisfying 
\[ d(x,y)>\varepsilon \Rightarrow d'(f(x),f(y)) > \delta .\]
Then $\#_{\mathrm{sep}}(X,d,\varepsilon)\leq \#_{\mathrm{sep}}(Y,d',\delta)$.
\end{lemma}
We will also need the following.
\begin{example} \label{example: separated set of Banach ball}
Let $(V,\norm{\cdot})$ be a real $n$-dimensional Banach space.
Let $B_r(V)$ be the closed $r$-ball of $V$ around the origin. For any $\varepsilon>0$
\[ \#_{\mathrm{sep}}(B_r(V),\norm{\cdot},\varepsilon) \leq \left(\frac{\varepsilon+2r}{\varepsilon}\right)^n.\]
\end{example}
\begin{proof}
Let $\mu$ be the Lebesgue measure (i.e. the Haar measure with respect to the translation) 
on $V$ normalized by $\mu(B_1(V))=1$.
For any $r>0$ we have 
$\mu(B_r(V)) = r^n$.
Choose $\{x_1,\dots,x_N\}\subset B_r(V)$ with $\norm{x_i-x_j}>\varepsilon$ for $i\neq j$.
Let $B_i$ be the closed $\varepsilon/2$-ball around $x_i$.
These $B_i$ are disjoint and contained in $B_{r+\varepsilon/2}(V)$.
Hence 
\[  N(\varepsilon/2)^n = \mu\left(\bigcup_{i=1}^N B_i\right) \leq \mu(B_{r+\varepsilon/2}(V)) = (r+\varepsilon/2)^n.\]
\end{proof}

Consider the complex plane $\affine$. 
For $a\in \affine$ and $r\geq 0$ we define $D_r(a)$ as the closed $r$-ball around $a$ in the plane.
We abbreviate $D_r(0)$ as $D_r$.
For $\Omega\subset \affine$ and $r>0$ we define $\partial_r\Omega$ as the set of $a\in \affine$
such that $D_r(a)$ non-trivially intersects with both $\Omega$ and $\affine \setminus \Omega$.
A sequence $\{\Omega_n\}_{n\geq 1}$ of bounded Borel subsets of $\affine$ is called a \textbf{F{\o}lner sequence} if for any $r>0$
\[ \lim_{n\to \infty}\frac{\area(\partial_r\Omega_n)}{\area(\Omega_n)} = 0 \quad 
   (\text{$\area(\cdot)$ is the standard Lebesgue measure}).\]
For example $\Omega_n := D_n$ is a F{\o}lner sequence. $\Omega_n := [0,n]^2$ is also.
The next lemma (which holds in a more general context of amenable groups) 
is a basis of the definition of mean dimension.
This was originally found by Ornstein--Weiss \cite[Chapter I, Sections 2 and 3]{Ornstein--Weiss}.
This formulation is due to Gromov \cite[p. 336]{Gromov}.
\begin{lemma}[Ornstein--Weiss lemma] \label{lemma: Ornstein--Weiss}
Let $h: \{\text{bounded Borel subsets of }\,\affine\}\to \mathbb{R}$ be a non-negative function satisfying 
the following three conditions.

\noindent 
(1) If $\Omega_1\subset \Omega_2$, then $h(\Omega_1)\leq h(\Omega_2)$.

\noindent 
(2) $h(\Omega_1\cup \Omega_2) \leq h(\Omega_1)+h(\Omega_2)$.

\noindent 
(3) For any $a\in \affine$ and any bounded Borel subset $\Omega \subset \affine$, we have 
$h(a+\Omega)=h(\Omega)$ where $a+\Omega:=\{a+z\in \affine| z\in \Omega\}$.

Then for any F{\o}lner sequence $\Omega_n$ $(n\geq 1)$ in $\affine$, the limit of the sequence 
\[ \frac{h(\Omega_n)}{\area(\Omega_n)} \quad (n\geq 1) \]
exists, and its value is independent of the choice of a F{\o}lner sequence.
\end{lemma}

Choose a F{\o}lner sequence $\{\Omega_n\}_{n\geq 1}$ in $\affine$.
Let $(X,d)$ be a compact metric space, and
suppose the group $\affine$ continuously acts on $X$.
For a subset $\Omega$ of $\affine$ we define a new distance $d_\Omega$ on $X$ by 
\[ d_\Omega (x,y) = \sup_{a\in \Omega} d(a.x,a.y) \quad (x,y\in X).\]
For any $\varepsilon>0$ the function $h(\Omega) := \widim_\varepsilon (X,d_\Omega)$ satisfies 
the three conditions in the Ornstein--Weiss lemma (Lemma \ref{lemma: Ornstein--Weiss}).
Then we define the \textbf{mean dimension} $\dim(X:\mathbb{C})$ by 
\[ \dim(X:\affine) := \lim_{\varepsilon\to 0} \left(\lim_{n\to \infty} 
   \frac{\widim_\varepsilon(X,d_{\Omega_n})}{\area(\Omega_n)}\right).\]
The value of $\dim(X:\affine)$ is independent of the choice of a distance $d$ on $X$ compatible with the topology.
So the mean dimension is a topological invariant.

Next we introduce \textbf{metric mean dimension} (Lindenstrauss--Weiss \cite[Section 4]{Lindenstrauss--Weiss}).
For any $\varepsilon >0$ the function $h(\Omega) := \log \#(X,d_{\Omega},\varepsilon)$ also satisfies the conditions of 
the Ornstein--Weiss lemma. Then we define the ``entropy at the scale $\varepsilon$'' by 
\[ S(X,d,\varepsilon) := \lim_{n\to \infty} \frac{\log \#(X,d_{\Omega_n},\varepsilon)}{\area(\Omega_n)}.\]
The topological entropy $h_{\mathrm{top}}(X:\affine)$ is the limit of $S(X,d,\varepsilon)$ as $\varepsilon\to 0$.
We define the metric mean dimension by 
\[ \mmdim(X,d:\mathbb{C}) := \liminf_{\varepsilon \to 0}\frac{S(X,d,\varepsilon)}{|\log \varepsilon|}.\]
This \textit{depends} on the choice of a distance.
The next theorem is fundamental (\cite[Theorem 4.2]{Lindenstrauss--Weiss}).
\begin{theorem}[Lindenstrauss--Weiss theorem]\label{thm: Lindenstrauss--Weiss}
Metric mean dimension is an upper bound on mean dimension:
\[ \dim(X:\affine) \leq \mmdim(X,d:\affine).\]
\end{theorem}
Therefore we can approach to an upper bound on mean dimension via metric mean dimension.
A difficulty of mean dimension lies in its \textit{global} nature.
We need to construct an $\varepsilon$-embedding from $X$ to a simplicial complex for an upper bound on 
$\dim(X:\affine)$.
But in many situations (in particular the case of $X=\moduli(\affine P^N)$) the space $X$ is a mysterious infinite dimensional 
space, and it is highly nontrivial to construct an $\varepsilon$-embedding 
from $X$ to an appropriate simplicial complex.
We can relax this difficulty by using metric mean dimension because metric mean dimension 
is a more \textit{local} quantity.
Its local nature is probably not obvious in the above definition.
The next lemma provides an explicit formulation.
Here we use the notation $D_R(\Lambda) := \bigcup_{a\in \Lambda} D_R(a)$ for $R\geq 0$ and $\Lambda\subset \affine$.

\begin{lemma} \label{lemma: local metric mean dimension}
For any $\delta>0$ and $R\geq 0$ the metric mean dimension $\mmdim(X,d:\affine)$ is equal to 
\begin{equation} \label{eq: local formula of metric mean dimension}
 \liminf_{\varepsilon\to 0}
 \left\{ 
 \left(\limsup_{L\to \infty} \frac{\sup_{x\in X} 
  \log \#(B_\delta(x,d_{D_R([0,L]^2)}), d_{[0,L]^2},\varepsilon)}{L^2}\right)/|\log\varepsilon|
 \right\}.
\end{equation}
Here $B_\delta(x,d_{D_R([0,L]^2)})$ is the closed $\delta$-ball 
around $x$ with respect to the distance $d_{D_R([0,L]^2)}$.
Indeed we can replace $\limsup_{L\to \infty}$ with $\lim_{L\to \infty}$.
But we don't need this fact.
\end{lemma}
\begin{proof}
We denote the right-hand side by $\mmdim'(X,d:\affine)$.
It is obvious that $\mmdim(X,d:\affine)\geq \mmdim'(X,d:\affine)$.
We can choose $x_1,\dots,x_n\in X$ such that 
\[ X= \bigcup_{i=1}^n B_\delta(x_i, d_{D_R([0,L]^2)}), \quad n\leq \#(X,d_{D_R([0,L]^2)},\delta).\]
Then $\#(X,d_{[0,L]^2},\varepsilon)$ is bounded by 
\[ \sum_{i=1}^n\#(B_\delta(x_i,d_{D_R([0,L]^2)}),d_{[0,L]^2},\varepsilon) 
   \leq n \sup_{x\in X} \#(B_\delta(x,d_{D_R([0,L]^2)}),d_{[0,L]^2},\varepsilon).\]
Therefore 
\[ \frac{\log\#(X,d_{[0,L]^2},\varepsilon)}{L^2}\leq \frac{\log\#(X,d_{D_R([0,L]^2)},\delta)}{L^2} 
   + \frac{\sup_{x\in X}\log\#(B_\delta(x,d_{D_R([0,L]^2)}),d_{[0,L]^2},\varepsilon)}{L^2}.\]
Letting $L\to \infty$, we get 
\[ S(X,d,\varepsilon) \leq S(X,d,\delta)  
   + \limsup_{L\to \infty} \frac{\sup_{x\in X} \log \#(B_\delta(x,d_{D_R([0,L]^2)}), d_{[0,L]^2},\varepsilon)}{L^2}.\]
Divide this by $|\log\varepsilon|$, and let $\varepsilon \to 0$. Then we get $\mmdim(X,d:\affine) \leq \mmdim'(X,d:\affine)$.
\end{proof}

The formula (\ref{eq: local formula of metric mean dimension}) probably looks complicated.
The point is that we only need to estimate $\#(B_\delta(x,d_{D_R([0,L]^2)}), d_{[0,L]^2},\varepsilon)$ for the calculation of the 
metric mean dimension $\mmdim(X,d:\affine)$.
This is a much more local problem than constructing an $\varepsilon$-embedding.
Probably we should also emphasize that metric mean dimension is \textit{not} a completely 
local quantity.
The term $\sup_{x\in X}$ in the formula 
(\ref{eq: local formula of metric mean dimension})
has a global nature; we need a uniform estimate all over $X$.
This requires us a detailed quantitative study of $X$, and it will be the main 
technical issue of the paper.

\begin{remark}
Some readers might think that the proof of the Lindenstrauss--Weiss theorem 
(Theorem \ref{thm: Lindenstrauss--Weiss}) produces an $\varepsilon$-embedding from the information of metric mean dimension.
This is true. But it does not provide a \textit{concrete} method to construct an $\varepsilon$-embedding.
The proof of \cite[Theorem 4.2]{Lindenstrauss--Weiss} is based on a probabilistic argument, and we cannot 
figure out a deterministic algorithm from it.
\end{remark}

\subsection{Energy density} \label{subsection: energy density}

Here we prepare some facts on the energy density (\ref{eq: energy density}).
The main result (Lemma \ref{lemma: another form of energy density} below)
is a slightly tricky exchange of the supremum and limit in the definition of $\rho(\affine P^N)$.
For a positive number $\lambda$ we define $\moduli_\lambda(\affine P^N)$ as the space of 
$\lambda$-Lipschitz (i.e. $|df|\leq \lambda$) holomorphic curves $f:\affine \to \affine P^N$.
There is a natural one-to-one correspondence between $\moduli(\affine P^N)$ and $\moduli_\lambda(\affine P^N)$: 
\begin{equation} \label{eq: moduli and moduli_lambda}
   \moduli(\affine P^N)\to \moduli_\lambda (\affine P^N), \quad f(z) \mapsto f(\lambda z).
\end{equation}
For $f\in \moduli_\lambda(\affine P^N)$ we define its energy density $\rho(f)$ by 
\[ \rho(f) := \lim_{n\to \infty} \left(\frac{1}{\area(\Omega_n)}\sup_{a\in \affine} \int_{a+\Omega_n} |df|^2 dxdy\right),\]
where $\{\Omega_n\}_{n\geq 1}$ is a F{\o}lner sequence in $\affine$.
This limit exists because of the Ornstein--Weiss lemma (Lemma \ref{lemma: Ornstein--Weiss}).
In particular 
\begin{equation*}
  \rho(f) = \lim_{R\to \infty} \left(\frac{1}{\pi R^2} \sup_{a\in \affine} \int_{|z-a|<R}|df|^2 dxdy\right)
          = \lim_{L\to \infty} \left(\frac{1}{L^2} \sup_{a\in \affine} \int_{a+[0,L]^2}|df|^2 dxdy\right).
\end{equation*}
For a Brody curve $f:\affine \to \affine P^N$ we have $\rho(f(\lambda z))= \lambda^2 \rho(f)$.
Hence from (\ref{eq: moduli and moduli_lambda}) 
\[ \sup_{f\in \moduli_\lambda(\affine P^N)}\rho(f) = \lambda^2 \rho(\affine P^N).\]
In \cite[Theorem 1.3]{Tsukamoto energy density} we proved the next lemma.
\begin{lemma} \label{lemma: another form of energy density}
\[ \lambda^2 \rho(\affine P^N) = \sup_{f\in \moduli_\lambda(\affine P^N)} \rho(f) 
 = \lim_{L\to \infty} \left(\frac{1}{L^2} \sup_{f\in \moduli_\lambda(\affine P^N)} \int_{[0,L]^2} |df|^2 dxdy\right).\]
\end{lemma}

\subsection{Notations} \label{subsection: notations}

Here we gather some frequently used notations.

$\bullet$
The number $N$ (which is the complex dimension of $\affine P^N$) is fixed throughout the paper.
So we treat it as a universal constant.
For two quantities $x$ and $y$ we write 
\[ x\lesssim y \]
if there exists a universal positive constant $C$ satisfying $x\leq Cy$. 
We also write 
\[ x \lesssim_{a,b,c,\dots,k} y \]
if there exists a positive constant $C(a,b,c,\dots,k)$ which depends only on the parameters $a,b,c,\dots,k$ satisfying 
$x\leq C(a,b,c,\dots,k)y$.

$\bullet$
For a non-zero $u\in \affine^{N+1}$ we denote by $[u]$ the point in $\affine P^N$ corresponding to $u$.
For two points $u,v\in \affine^{N+1} \setminus \{0\}$, the distance between $[u]$ and $[v]$ with respect to the Fubini--Study 
metric is given by 
\[ d_{\mathrm{FS}}([u],[v]) 
   = \frac{1}{\sqrt{\pi}}\arccos \frac{|\langle u,v\rangle|}{|u|\, |v|} \in \left[0,\frac{\sqrt{\pi}}{2}\right].\]
Here $\langle u,v\rangle$ is the standard Hermitian inner product, and 
we choose the branch of $\arccos$ satisfying $\arccos 0 = \pi/2$ and $\arccos 1 = 0$.
This is a conceptually nice distance. But it is not convenient for concrete calculations.
So we introduce another distance by 
\[ d([u],[v]) := \sin\left(\sqrt{\pi}d_{\mathrm{FS}}([u],[v])\right) =  \frac{|u\wedge v|}{|u|\, |v|},\]
where $u\wedge v \in \Lambda^2(\affine^{N+1})$ is the exterior product of $u$ and $v$.
The two distances $d_{\mathrm{FS}}$ and $d$ are Lipschitz equivalent.
For $u=(1,z)$ and $v=(1,w)$ with $z,w\in \affine^N$ 
\begin{equation} \label{eq: chord distance formula} 
   d([1:z], [1:w]) = \frac{\sqrt{|z-w|^2+|z\wedge w|^2}}{\sqrt{1+|z|^2} \sqrt{1+|w|^2}}.
\end{equation}
When $N=1$, the term $z\wedge w$ vanishes 
and this is the chord distance between two points $z$ and $w$ in the Riemann sphere $\affine \cup \{\infty\}$.

$\bullet$ 
Let $f,g:\affine \to \affine P^N$ be two holomorphic curves.
For a subset $A\subset \affine$ we set 
\[ \mathbf{d}_A(f,g) := \sup_{z\in A} d(f(z),g(z)).  \]

\section{Proof of the main theorem} \label{section: proof of the main theorem}

In this section we prove Theorem \ref{thm: main theorem}.
Our proof is based on two propositions
(Propositions \ref{prop: resolution of singularity} and \ref{prop: local study around nondegenerate curve})
whose proofs occupy the rest of the papers.
A crucial ingredient of the proof is the following notion.
\begin{definition} \label{def: nondegenerate curve}
Let $R$ be a positive number, and let $\Lambda$ be a subset of $\affine$.
A holomorphic curve $f:\affine \to \affine P^N$ is said to be $R$-nondegenerate over $\Lambda$ if it satisfies 
\[ \forall a\in \Lambda: \quad \norm{df}_{L^\infty(D_R(a))} \geq 1/R.\]
\end{definition}
This is a quantitative version of an old idea of Yosida \cite{Yosida}.
In \cite{Yosida}
a meromorphic function $f\in \moduli(\affine P^1)$ is said to be of first category if 
it is $R$-nondegenerate all over the plane for some $R>0$.
This is equivalent to the condition that 
the closure of the $\affine$-orbit of $f$ in $\moduli(\affine P^1)$ does not contain a constant function.
This idea of Yosida played a quite important role in \cite{Matsuo--Tsukamoto Brody curves}
for the proof of the lower bound $\dim(\moduli(\affine P^N):\affine)\geq 2(N+1)\rho(\affine P^N)$.
Intuitively speaking,
Yosida's condition is a kind of ``transversality'', and 
Definition \ref{def: nondegenerate curve} corresponds to a ``quantitative transversality''.
The system $\moduli(\affine P^N)$ is non-singular around nondegenerate curves in the sense that
if $f\in \moduli(\affine P^N)$ is $R$-nondegenerate over $\Lambda$, then a neighborhood of 
$f$ (whose size depends on $R$ and $\Lambda$) can be described by a first-order 
deformation technique.
This is the reason of the importance of the notion.

Unfortunately some Brody curves are degenerate, and they become singularities for our analysis.
Our first main task is to resolve these singularities, and
the next proposition establishes a quantitative resolution of singularities for our purpose.
Recall the notation: $D_R(\Lambda) = \bigcup_{a\in \Lambda} D_R(a)$ for $R\geq 0$ and $\Lambda\subset \affine$.

\begin{proposition} \label{prop: resolution of singularity}
There exist $\delta_1>0$ and $C_1>1$ satisfying the following statement.
For any $\lambda>1$ we can choose $R_1=R_1(\lambda)>0$ such that for any $f\in \moduli(\affine P^N)$ and 
any bounded set $\Lambda\subset \affine$ we can construct a map
\[ \Phi: \{g\in \moduli(\affine P^N)|\, \mathbf{d}_{D_{R_1}(\Lambda)}(f,g)\leq \delta_1\} \to \moduli_\lambda(\affine P^N) \]
satisfying the following conditions.
(Recall that $\moduli_\lambda(\affine P^N)$ is the space of $\lambda$-Lipschitz holomorphic curves.)

\noindent 
(1) $\Phi(f)$ is $R_1$-nondegenerate over $\Lambda$.

\noindent 
(2) Let $g_1,g_2\in \moduli(\affine P^N)$ with 
$\mathbf{d}_{D_{R_1}(\Lambda)}(f,g_1)\leq \delta_1$ and $\mathbf{d}_{D_{R_1}(\Lambda)}(f,g_2)\leq \delta_1$.
For any $z\in \affine$
\begin{equation} \label{eq: distorsion of blow up}
   C_1^{-1}d(\Phi(g_1)(z), \Phi(g_2)(z))\leq 
   d(g_1(z),g_2(z)) \leq C_1 \sup_{|w-z|\leq 3} d(\Phi(g_1)(w),\Phi(g_2)(w)).
\end{equation}
\end{proposition}

The next proposition is a conclusion of our quantitative study of $\moduli_\lambda(\affine P^N)$
around nondegenerate curves.

\begin{proposition} \label{prop: local study around nondegenerate curve}
For any $R>0$ and $0<\varepsilon<1$ there exist positive numbers $\delta_2=\delta_2(R)$, 
$C_2=C_2(R)$ and $C_3=C_3(\varepsilon)$ satisfying the following statement.
Let $f\in \moduli_2(\affine P^N)$, and let $\Lambda\subset \affine$ be a square of side length $L\geq 1$.
Suppose $f$ is $R$-nondegenerate over $\Lambda$. Then 
\begin{equation*}
  \#_{\mathrm{sep}}\left(\{g\in \moduli_2(\affine P^N)|\, \mathbf{d}_{D_5(\Lambda)}(f,g)\leq \delta_2\}, 
  \mathbf{d}_{\Lambda},\varepsilon\right) 
  \leq (C_2/\varepsilon)^{2(N+1)\int_{\Lambda}|df|^2 dxdy + C_3 L}.
\end{equation*}
\end{proposition}

Assuming Propositions \ref{prop: resolution of singularity} and \ref{prop: local study around nondegenerate curve},
we prove the main theorem.

\begin{proof}[Proof of Theorem \ref{thm: main theorem}]
The lower bound $\dim(\moduli(\affine P^N):\affine)\geq 2(N+1)\rho(\affine P^N)$ was proved 
in \cite{Matsuo--Tsukamoto Brody curves}.
So the problem is the upper bound.
We introduce a distance on $\moduli(\affine P^N)$ by 
$\dist(f,g) := \sup_{|z|\leq 1}d(f(z),g(z))$.
Because of the unique continuation principle, this is a distance compatible with the compact-open topology.
We will prove the upper bound on the metric mean dimension:
$\mmdim(\moduli(\affine P^N),\dist:\affine) \leq 2(N+1)\rho(\affine P^N)$.
Then we will get $\dim(\moduli(\affine P^N):\affine) \leq 2(N+1)\rho(\affine P^N)$ by the Lindenstrauss--Weiss theorem 
(Theorem \ref{thm: Lindenstrauss--Weiss}).

Take $1<\lambda<2$ and $0<\varepsilon <1$.
Let $R_1=R_1(\lambda)>0$ be the constant introduced in Proposition \ref{prop: resolution of singularity}.
For $R_1$ and $\varepsilon$, we take $\delta_2=\delta_2(R_1)$, $C_2=C_2(R_1)$ 
and $C_3=C_3(\varepsilon)>0$ (the positive constants introduced in Proposition 
\ref{prop: local study around nondegenerate curve}).

Take a positive number $L$ and $f\in \moduli(\affine P^N)$.
We apply Proposition \ref{prop: resolution of singularity} to $f$ and $\Lambda:=[-4,L+4]^2$.
Then we get the map $\Phi$ from $\{g\in \moduli(\affine P^N)|\, \mathbf{d}_{D_{R_1}(\Lambda)}(f,g)\leq \delta_1\}$ to 
$\moduli_\lambda(\affine P^N)$ satisfying the conditions (1) and (2) in Proposition \ref{prop: resolution of singularity}.
We set $\delta =\delta(R_1) = \min(\delta_1, \delta_2/C_1)$ 
(where $\delta_1>0$ and $C_1>1$ are the constants introduced in Proposition \ref{prop: resolution of singularity})
and consider the ball 
$B_\delta(f,\dist_{D_{R_1+10}([0,L]^2)})$ of radius $\delta$ around $f$ in $\moduli(\affine P^N)$ 
with respect to $\dist_{D_{R_1+10}([0,L]^2)}$. This is contained in the set 
$\{g\in \moduli(\affine P^N)|\, \mathbf{d}_{D_{R_1}(\Lambda)}(f,g)\leq \delta_1\}$.
For $g\in B_\delta(f, \dist_{D_{R_1+10}([0,L]^2)})$, by (\ref{eq: distorsion of blow up}) 
\begin{equation*}
   \mathbf{d}_{D_5(\Lambda)}(\Phi(f),\Phi(g)) 
   \leq C_1 \mathbf{d}_{D_5(\Lambda)}(f,g) 
   \leq C_1 \dist_{D_{R_1+10}([0,L]^2)}(f,g) \leq \delta_2.
\end{equation*}
Hence $\Phi(B_\delta(f,\dist_{D_{R_1+10}([0,L]^2)}))$ is contained in 
$\{g\in \moduli_2(\affine P^N)|\, \mathbf{d}_{D_5(\Lambda)}(\Phi(f),g)\leq \delta_2\}$.
Then by applying Proposition \ref{prop: local study around nondegenerate curve} to 
$\Phi(f)$ and $\Lambda = [-4,L+4]^2$ (note that $\Phi(f)$ is $R_1$-nondegenerate over $\Lambda$),
\[ \#_{\mathrm{sep}}(\Phi(B_\delta(f,\dist_{D_{R_1+10}([0,L]^2)})), \mathbf{d}_{\Lambda},\varepsilon/(3C_1))
  \leq (3C_1 C_2/\varepsilon)^{2(N+1)\int_{\Lambda}|d\Phi(f)|^2 dxdy + C_3(L+8)}. \]
By (\ref{eq: distorsion of blow up}), for $g_1,g_2\in B_\delta(f,\dist_{D_{R_1+10}([0,L]^2)})$
\[ \dist_{[0,L]^2}(g_1,g_2) \leq C_1\mathbf{d}_{\Lambda}(\Phi(g_1),\Phi(g_2)).  \]
Hence $\dist_{[0,L]^2}(g_1,g_2) > \varepsilon/3$ implies $\mathbf{d}_\Lambda(\Phi(g_1),\Phi(g_2)) > \varepsilon/3C_1$.
By (\ref{eq: separated set and spanning set}) and  Lemma \ref{lemma: separated set}
\begin{equation*}
   \begin{split}
   \#(B_\delta(f,\dist_{D_{R_1+10}([0,L]^2)}),\dist_{[0,L]^2},\varepsilon)
   &\leq \#_{\mathrm{sep}}(B_\delta(f,\dist_{D_{R_1+10}([0,L]^2)}),\dist_{[0,L]^2},\varepsilon/3) \\
   &\leq \#_{\mathrm{sep}}(\Phi(B_\delta(f,\dist_{D_{R_1+10}([0,L]^2)})), \mathbf{d}_{\Lambda},\varepsilon/(3C_1)) \\
   &\leq (3 C_1 C_2/\varepsilon)^{2(N+1)\int_{\Lambda}|d\Phi(f)|^2 dxdy + C_3(L+8)}.  
   \end{split}
\end{equation*}
Hence the supremum of $\log \#(B_\delta(f,\dist_{D_{R_1+10}([0,L]^2)}),\dist_{[0,L]^2},\varepsilon)$ over 
$f\in \moduli(\affine P^N)$ is bounded by (note $\Phi(f)\in \moduli_\lambda(\affine P^N)$)
\[ (\const_{R_1} +|\log\varepsilon|)\left(2(N+1)\sup_{g\in \moduli_\lambda(\affine P^N)}
   \int_{[0,L]^2}|dg|^2 dxdy + \const_\varepsilon\cdot L + \const_\varepsilon\right).\]
By Lemma \ref{lemma: another form of energy density} 
\[ \lim_{L\to \infty}\left(\frac{1}{L^2}
   \sup_{g\in \moduli_\lambda(\affine P^N)} \int_{[0,L]^2}|dg|^2 dxdy\right) = \lambda^2\rho(\affine P^N).\]
Therefore 
\begin{equation*}
  \begin{split}
   \limsup_{L\to \infty} &\frac{\sup_{f\in \moduli(\affine P^N)}
    \log \#(B_\delta(f,\dist_{D_{R_1+10}([0,L]^2)}),\dist_{[0,L]^2},\varepsilon)}{L^2}\\
   &\leq 
   (\const_{R_1} + |\log\varepsilon|) 2(N+1)\lambda^2 \rho(\affine P^N).
  \end{split}
\end{equation*}
Divide this by $|\log \varepsilon|$ and let $\varepsilon \to 0$. 
Then by the formula (\ref{eq: local formula of metric mean dimension}) in Lemma \ref{lemma: local metric mean dimension}
\[ \dim(\moduli(\affine P^N), \dist:\affine) \leq 2(N+1)\lambda^2 \rho(\affine P^N).\]
Here $\lambda$ is an arbitrary number between $1$ and $2$. So we can let $\lambda\to 1$ and conclude  
$\dim(\moduli(\affine P^N),\dist:\affine) \leq 2(N+1)\rho(\affine P^N)$.
\end{proof}

\section{Blowing up degenerate curves} \label{section: blowing up degenerate curves}

We prove Proposition \ref{prop: resolution of singularity} in this section.
Here is the idea.
If a holomorphic curve $f$ is degenerate around some point $p\in \affine$, then 
we ``blow up'' it by gluing a sufficiently concentrated rational curve at the point $p$.
Then the resulting curve $\hat{f}$ is nondegenerate around $p$.
We repeatedly apply this procedure to $f$ until all the degenerate regions are eliminated.
This technique was first developed in \cite{Matsuo--Tsukamoto Brody curves} for a single curve $f$.
Here we need to apply it to a family of holomorphic curves.
A main new issue is to analyze the changes of the metric structure of the family under blow-ups.
We denote by $\mathcal{H}(\affine P^N)$ the set of all holomorphic curves 
$f:\affine \to \affine P^N$.
The following lemma studies the effect of one blow-up in detail.

\begin{lemma} \label{lemma: one blow up}
We can choose $0<\delta_3<1$, $R_2>0$ and $C_4>1$ so that for any $p\in \affine$, $q\in \affine P^N$
and $R\geq R_2$ there exists a map 
\[ \Psi: \{f\in \mathcal{H}(\affine P^N)|\, f(D_{R}(p))\subset B_{\delta_3}(q)\}\to \mathcal{H}(\affine P^N), \quad 
   f\mapsto \hat{f}, \]
satisfying the following conditions.

\noindent 
(1) For any holomorphic curve $f:\affine \to \affine P^N$ with $f(D_R(p))\subset B_{\delta_3}(q)$
(i.e. $d(q,f(z))\leq \delta_3$ for all $z\in D_{R}(p)$), the curve $\hat{f}$ satisfies 
\begin{align}
  &d(f(z),\hat{f}(z)) \leq \frac{C_4}{|z-p|^3},  \label{eq: distance between f and f hat}\\
  &\frac{1}{100} < \norm{d\hat{f}}_{L^\infty(D_{R/2}(p))} <1,  \label{eq: nondegeneracy of f hat}\\
  &\left||df|(z)-|d\hat{f}|(z)\right| 
  \leq \frac{C_4}{|z-p|^3}|df|(z) + \frac{C_4}{|z-p|^4} \quad (|z-p|\geq 1). \label{eq: difference between df and df hat} 
\end{align}
  
\noindent 
(2) Let $f,g:\affine\to \affine P^N$ be holomorphic curves satisfying 
$f(D_{R}(p))\subset B_{\delta_3}(q)$ and $g(D_{R}(p))\subset B_{\delta_3}(q)$.
For $|z-p|\leq 1$, the curves $\hat{f}$ and $\hat{g}$ satisfy 
\begin{equation} \label{eq: distance between f and g near the center}
   C_4^{-1} d(\hat{f}(z),\hat{g}(z))\leq  d(f(z),g(z)) \leq 
   C_4 \sup_{|w-p|\leq 2} d(\hat{f}(w),\hat{g}(w)).
\end{equation}
For $|z-p|\geq 1$ they satisfy 
\begin{equation} \label{eq: distance between f and g far from the center}
  \begin{split}
   d(\hat{f}(z),\hat{g}(z)) \leq \left(1+\frac{C_4}{|z-p|^3}\right) d(f(z),g(z)), \\
   d(f(z),g(z)) \leq \left(1+\frac{C_4}{|z-p|^3}\right) d(\hat{f}(z),\hat{g}(z)).
  \end{split}
\end{equation}  
\end{lemma}

\begin{proof}
We can assume $p=0$ and $q=[1:0:\dots:0]$ from the symmetry.
The positive constants $\delta_3$, $R_2$ and $C_4$ are chosen so that 
\[ R_2 \gg 1, \quad \delta_3 \ll \frac{1}{R_2^4}, \quad C_4 \gg 1. \]
Specific conditions will be imposed on them through the argument.
Take a holomorphic curve $f:\affine \to \affine P^N$ with $f(D_{R})\subset B_{\delta_3}(q)$.
We write $f = [1:f_1:\dots:f_N]$ ($f_i$: meromorphic function) and set 
$F(z) :=(f_1(z),\dots,f_N(z))$.
Then $f(z)=[1:F(z)]$ and the spherical derivative $|df|$ is expressed by 
\[ |df|(z) = \frac{\sqrt{|F'(z)|^2 + |F(z)\wedge F'(z)|^2}}{\sqrt{\pi}(1+|F(z)|^2)}.\]
From $f(D_R)\subset B_{\delta_3}(q)$ with $\delta_3\ll 1$ and $R\geq R_2\gg 1$, we can assume 
\begin{equation} \label{eq: restriction on F}
  |F(z)| < 2\delta_3, \quad |F'(z)| < \delta_3 \quad (|z|\leq R/2).
\end{equation}

Define a holomorphic curve $h:\affine \to \affine P^N$ by 
$h(z) = [1:a/z^3:\dots:a/z^3]$. We choose a positive constant $a$ so that 
$\norm{dh}_{L^\infty(\affine)} = 1/10$.
Here we have (set $H(z) := (a/z^3,\dots,a/z^3)$)
\[ |dh|(z) = \frac{|H'(z)|}{\sqrt{\pi}(1+|H(z)|^2)} = \frac{3a|z|^2\sqrt{N}}{\sqrt{\pi}(|z|^6+Na^2)}.\] 
By $R_2\gg 1$ we can assume $\norm{dh}_{L^\infty(D_{R_2/2})} =1/10$.
We define a holomorphic curve 
$\hat{f}:\affine \to \affine P^N$ by setting $\hat{f}(z):= [1:f_1(z)+a/z^3:\dots:f_N(z)+a/z^3]$.
We set $\hat{F}(z) := F(z)+H(z)$. Then $\hat{f}=[1:\hat{F}]$.
We will show that the map $f\mapsto \hat{f}$ satisfies the required conditions.

First, by the distance formula (\ref{eq: chord distance formula}) 
\[ d(f(z),\hat{f}(z)) = \frac{\sqrt{|H|^2+ |F\wedge H|^2}}{\sqrt{1+|F|^2}\sqrt{1+|\hat{F}|^2}}
   \leq \frac{\sqrt{|H|^2+|F|^2|H|^2}}{\sqrt{1+|F|^2}} = |H| = \frac{a\sqrt{N}}{|z|^3}. \]
This shows (\ref{eq: distance between f and f hat}).
Next we consider (\ref{eq: nondegeneracy of f hat}) and (\ref{eq: difference between df and df hat}).
We have 
\[ |d\hat{f}| = \frac{\sqrt{|F'+H'|^2+|F\wedge F' + F\wedge H' + H\wedge F'|^2}}{\sqrt{\pi}(1+|F+H|^2)}.\]
For $|z|\leq R_2/2$, we can assume $|H'(z)|\geq 2\delta_3$ and $|H(z)|\geq 4\delta_3$ by $\delta_3 \ll 1/R_2^4$.
So by (\ref{eq: restriction on F}) 
\begin{equation} \label{eq: estimate on F hat near the origin}
   \bigl||F'+H'|-|H'|\bigr| \leq \delta_3 \leq \frac{|H'|}{2}, \quad 
   \bigl||F+H|-|H|\bigr| \leq 2\delta_3 \leq \frac{|H|}{2} \quad \text{on $D_{R_2/2}$}.
\end{equation}
Hence $|F'+H'|\geq |H'|/2$ and $|F+H|\leq 2|H|$. Then
\[ |d\hat{f}| \geq \frac{|H'|/2}{\sqrt{\pi}(1+4|H|^2)} \geq \frac{|dh|}{8}.\]
This implies a half of (\ref{eq: nondegeneracy of f hat}):
$\norm{d\hat{f}}_{L^\infty(D_{R_2/2})}\geq \norm{dh}_{L^\infty(D_{R_2/2})}/8 = 1/80 >1/100$.
For $|z|\leq R_2/2$ we also have $|F'+H'|\leq 2|H'|$, $|F+H|\geq |H|/2$ and 
\[ |F\wedge F' + F\wedge H' + H\wedge F'|\leq |F| |F'| + |F| |H'| +|H| |F'|
   \leq 2\delta_3^2 + 2\delta_3 |H'| + \delta_3 |H| \leq |H'| \]
by (\ref{eq: restriction on F}), $\delta_3 \ll 1 $ and $R_2|H'(z)| \gtrsim |H(z)|\geq 2\delta_3$.
Therefore (recall $|dh|\leq 1/10$)
\[ |d\hat{f}|(z) \leq \frac{\sqrt{4|H'|^2+|H'|^2}}{\sqrt{\pi}(1+|H|^2/4)} \leq 
   \frac{4\sqrt{5}|H'|}{\sqrt{\pi}(1+|H|^2)} = 4\sqrt{5}|dh| < 1  .\]
Thus $1/100 < \norm{d\hat{f}}_{L^\infty(D_{R_2/2})} < 1$.
(Note that we have not yet finished to prove (\ref{eq: nondegeneracy of f hat}).)

From the triangle inequality 
\begin{equation*}
  \begin{split} 
   \bigl|\sqrt{|F'+H'|^2+|F\wedge F' + F\wedge H' + H\wedge F'|^2}-
   \sqrt{|F'|^2+|F\wedge F'|^2}\,\bigr|  \\
   \leq  \sqrt{|H'|^2+|F\wedge H'+H\wedge F'|^2} \leq |H'| + |F| |H'| + |H| |F'|.
  \end{split}
\end{equation*}
As we have
\[ \frac{1+|F|^2}{1+|\hat{F}|^2} = \frac{1+|\hat{F}|^2 -2\mathrm{Re}\langle \hat{F},H\rangle + |H|^2}{1+|\hat{F}|^2}
   = 1 + \frac{-2\mathrm{Re}\langle \hat{F},H\rangle + |H|^2}{1+|\hat{F}|^2},\]
we get
\begin{equation} \label{eq: difference between the denominators}
   \left|\frac{1}{1+|\hat{F}|^2} - \frac{1}{1+|F|^2}\right| \leq \frac{|H|+|H|^2}{1+|F|^2}.
\end{equation}
Then by a straightforward calculation, we can check that for $|z|\geq 1$ (note $|H'|\lesssim |H|\lesssim 1$) 
\[ \left| |d\hat{f}|(z) -|df|(z)\right| \lesssim  |H|\, |df|(z) + |H'|.\]
This shows (\ref{eq: difference between df and df hat}).
For $R_2/2\leq |z|\leq R/2$ we have $|df|(z) \lesssim \delta_3$ and 
\[ |d\hat{f}|(z) \lesssim (1+|H|) \delta_3 + |H'| \ll 1 \quad (\text{by $\delta_3\ll 1$ and $R_2\gg 1$}).\]
Therefore we have $|d\hat{f}|<1$ over $R_2/2\leq |z|\leq R/2$.
Combining this with $1/100 < \norm{df}_{L^\infty(D_{R_2/2})}<1$, we get (\ref{eq: nondegeneracy of f hat}).

Next we consider (2).
Let $f=[1:F]$ and $g=[1:G]$ be two holomorphic curves with 
$f(D_R)\subset B_{\delta_3}(q)$ and $g(D_R)\subset B_{\delta_3}(q)$.
We have $\hat{f} = [1:\hat{F}] = [1:F+H]$, $\hat{g} = [1:\hat{G}] = [1:G+H]$ and 
\[ d(\hat{f}(z), \hat{g}(z)) = \frac{\sqrt{|\hat{F}-\hat{G}|^2 + |\hat{F}\wedge \hat{G}|^2}}{\sqrt{1+|\hat{F}|^2}\sqrt{1+|\hat{G}|^2}},\]
\[ \hat{F}-\hat{G} = F-G, \quad \hat{F}\wedge \hat{G} = F\wedge G + (F-G)\wedge H.\]
Suppose $|z|\leq 1$. We have 
$|\hat{F}|\geq |H|/2$ and $|\hat{G}|\geq |H|/2$ by 
(\ref{eq: estimate on F hat near the origin}).
Then 
\begin{equation*}
   \begin{split}
   d(\hat{f}(z),\hat{g}(z)) &\leq 
   \frac{\sqrt{|F-G|^2+|F\wedge G|^2} + |F-G|\, |H|}{1+|H|^2/4} 
   \leq 2\sqrt{|F-G|^2+|F\wedge G|^2}  \\
   &\lesssim d(f(z),g(z)) \quad (\text{$|F(z)|, |G(z)|\leq 2\delta_3$ by (\ref{eq: restriction on F})}).
   \end{split}
\end{equation*}
This proves the first part of (\ref{eq: distance between f and g near the center}).
For $|z|\leq 1$ the Cauchy estimate shows
\[ |F(z)-G(z)|\lesssim \sup_{|w|=2}|F(w)-G(w)| = \sup_{|w|=2}|\hat{F}(w)-\hat{G}(w)|,\]
\begin{equation*}
   \begin{split}
   |F(z)\wedge G(z)| &\lesssim \sup_{|w|=2}|F(w)\wedge G(w)| \leq 
   \sup_{|w|=2} \left(|\hat{F}(w)\wedge \hat{G}(w)|+|\hat{F}(w)-\hat{G}(w)|\, |H(w)|\right) \\
   &\lesssim \sup_{|w|=2} \sqrt{|\hat{F}(w)-\hat{G}(w)|^2 + |\hat{F}(w)\wedge \hat{G}(w)|^2}
   \quad (\text{by $H(w)\lesssim 1$ on $|w|=2$}).
   \end{split}
\end{equation*}
Since $|\hat{F}|, |\hat{G}|\lesssim 1$ on $|w|=2$, we get 
\[ d(f(z),g(z)) \leq |F(z)-G(z)|+|F(z)\wedge G(z)| \lesssim \sup_{|w|=2}d(\hat{f}(w),\hat{g}(w)).\]
This finishes the proof of (\ref{eq: distance between f and g near the center}).

Suppose $|z|\geq 1$. By (\ref{eq: difference between the denominators}) 
\[ \frac{1}{1+|\hat{F}|^2}\leq \frac{1+\const\, |H|}{1+|F|^2}, \quad 
   \frac{1}{1+|\hat{G}|^2}\leq \frac{1+\const\, |H|}{1+|G|^2}.\] 
As we have 
\begin{equation*}
  \begin{split}
   \sqrt{|\hat{F}-\hat{G}|^2+|\hat{F}\wedge \hat{G}|^2} &\leq 
   \sqrt{|F-G|^2+|F\wedge G|^2} + |F-G|\, |H| \\
   &\leq (1+|H|)\sqrt{|F-G|^2+|F\wedge G|^2},
  \end{split}
\end{equation*}
\[ d(\hat{f}(z),\hat{g}(z)) \leq (1+\const\, |H|)^2\frac{\sqrt{|F-G|^2+|F\wedge G|^2}}{\sqrt{1+|F|^2}\sqrt{1+|G|^2}}
   \leq (1+\const\, |H|)d(f(z),g(z)).\]
Here we have used $|H|\lesssim 1$.
Similarly we get 
$d(f(z),g(z))\leq (1+\const\, |H|)d(\hat{f}(z),\hat{g}(z))$.
This proves (\ref{eq: distance between f and g far from the center}).
We have finished to prove the lemma.
\end{proof}

We restate Proposition \ref{prop: resolution of singularity} for the convenience of readers.

\begin{proposition}[= Proposition \ref{prop: resolution of singularity}]
For any $\lambda>1$ we can choose $R_1 = R_1(\lambda)>0$ such that for any $f\in \moduli(\affine P^N)$ and 
any bounded set $\Lambda\subset \affine$ we can construct a map 
\[ \Phi: \{g\in \moduli(\affine P^N)|\, \mathbf{d}_{D_{R_1}(\Lambda)}(f,g)\leq \delta_3/4\} \to \moduli_\lambda(\affine P^N) \]
satisfying the following conditions.

\noindent 
(1) $\Phi(f)$ is $R_1$-nondegenerate over $\Lambda$.

\noindent 
(2) Let $g_1,g_2\in \moduli(\affine P^N)$ with 
$\mathbf{d}_{D_{R_1}(\Lambda)}(f,g_1)\leq \delta_3/4$ and $\mathbf{d}_{D_{R_1}(\Lambda)}(f,g_2)\leq \delta_3/4$.
For any $z\in \affine$
\begin{equation*}
   d(\Phi(g_1)(z), \Phi(g_2)(z))\lesssim  d(g_1(z),g_2(z))
   \lesssim \sup_{|w-z|\leq 3} d(\Phi(g_1)(w),\Phi(g_2)(w)).
\end{equation*}
\end{proposition}

\begin{proof}
We can assume $1<\lambda<2$.
Let $R=R(\lambda)>R_2$ ($R_2$ is the constant introduced in Lemma \ref{lemma: one blow up}) be a large positive number. 
We choose $R$ so large that the conditions (\ref{eq: error far from centers 1}), 
(\ref{eq: error far from centers 2}) and (\ref{eq: error far from centers 3}) below are satisfied.
We will choose $R_1$ sufficiently larger than $R$ later.
Let $\{p_1,\dots,p_n\}$ be a maximal subset of $\Lambda$ satisfying $|p_i-p_j|>2R$ $(i\neq j)$.
Then $\Lambda$ is contained in the union of the disks $D_{2R}(p_i)$.
Since $R\gg 1$, we can assume that for all $z\in \affine$
\begin{align}
    &\sum_{p_i:|z-p_i|>R}\frac{C_4}{|z-p_i|^3} < \frac{\delta_3}{2},  \label{eq: error far from centers 1}\\
    &\left(1+\sum_{p_i: |z-p_i|>R/2}\frac{C_4}{|z-p_i|^4}\right)
    \prod_{p_i: |z-p_i|>R/2}\left(1+\frac{C_4}{|z-p_i|^3}\right) 
    < \lambda <2, \label{eq: error far from centers 2} 
\end{align}
\begin{equation} \label{eq: error far from centers 3}
  \begin{split}
   \min\left(\frac{\delta_3}{4R\sqrt{\pi}},\frac{1}{100}\right)&\prod_{p_i:|p_i-z|>R}\left(1-\frac{C_4}{|z-p_j|^3}\right) \\
   &-\left(\sum_{p_i: |z-p_i|>R}\frac{C_4}{|z-p_i|^4}\right)
    \prod_{p_i: |z-p_i|>R}\left(1+\frac{C_4}{|z-p_i|^3}\right) > \frac{\delta_3}{10R}.
  \end{split}
\end{equation}
Here $0<\delta_3<1$ and $C_4>1$ are the constants introduced in Lemma \ref{lemma: one blow up}.

We denote by $X$ the set of $g\in \moduli(\affine P^N)$ satisfying $\mathbf{d}_{D_{R(\Lambda)}}(f,g)\leq \delta_3/4$.
We call the point $p_i$ good if $\norm{df}_{L^\infty(D_R(p_i))}\geq \delta_3/(4R\sqrt{\pi})$.
Otherwise we call it bad.
If $p_i$ is bad, then $\norm{df}_{L^\infty(D_R(p_i))} < \delta_3/(4R\sqrt{\pi})$ and hence 
$f(D_R(p_i))\subset B_{\delta_3/4}(f(p_i))$.
This implies 
\begin{equation} \label{eq: behavior around bad points}
   \forall g\in X: \quad g(D_R(p_i))\subset B_{\delta_3/2}(f(p_i)) \quad \text{for bad $p_i$}.
\end{equation}
We will blow up each curve $g\in X$ around all bad points $p_i$ by using Lemma \ref{lemma: one blow up}
repeatedly.

We inductively construct maps $\Phi_i:X\to \mathcal{H}(\affine P^N)$ for $0\leq i\leq n$ satisfying 
\begin{equation} \label{eq: induction hypothesis}
  d(g(z),\Phi_i(g)(z)) \leq \sum_{j=1}^i \frac{C_4}{|z-p_j|^3} \quad (\forall g\in X, z\in \affine).
\end{equation}
We define $\Phi_0$ as the identity map.
Suppose we have already constructed $\Phi_i$. 
If $p_{i+1}$ is good, then we set $\Phi_{i+1} := \Phi_i$.
If $p_{i+1}$ is bad, then we proceed as follows.
For $g\in X$ and $z\in D_R(p_{i+1})$
\[ d(g(z),\Phi_i(g)(z))\leq \sum_{j=1}^i \frac{C_4}{|z-p_j|^3} < \frac{\delta_3}{2} \quad 
   \text{by (\ref{eq: error far from centers 1}) and (\ref{eq: induction hypothesis})}.\]
Then by (\ref{eq: behavior around bad points}) the image of $D_R(p_{i+1})$ under the map $\Phi_i(g)$ is 
contained in $B_{\delta_3}(f(p_{i+1}))$.
From Lemma \ref{lemma: one blow up} we can construct a map 
\[ \Psi: \{h\in \mathcal{H}(\affine P^N)|\, h(D_R(p_{i+1}))\subset B_{\delta_3}(f(p_{i+1}))\}\to \mathcal{H}(\affine P^N),\quad 
   h \mapsto \hat{h}, \]
satisfying the conditions (1) and (2) of Lemma \ref{lemma: one blow up}.
Then we set $\Phi_{i+1} := \Psi\circ \Phi_i$.
This satisfies the following. Let $g\in X$.
\begin{itemize}
  \item $d(\Phi_i(g)(z),\Phi_{i+1}(g)(z)) \leq C_4/|z-p_{i+1}|^3$. 
  (Then $\Phi_{i+1}$ satisfies the condition (\ref{eq: induction hypothesis}). So we can continue the induction process.)
  \item $1/100 < \norm{d\Phi_{i+1}(g)}_{L^\infty(D_{R/2}(p_{i+1}))} < 1$.
  \item For $|z-p_{i+1}|\geq 1$
  \begin{equation}  \label{eq: difference between spherical derivatives under induction}
   \bigl| |d\Phi_{i+1}(g)|-|d\Phi_i(g)|\bigr| \leq \frac{C_4}{|z-p_{i+1}|^3}|d\Phi_i(g)| + \frac{C_4}{|z-p_{i+1}|^4}.
  \end{equation}
  \item Let $g_1,g_2\in X$. For $|z-p_{i+1}|\leq 1$
  \begin{equation}  \label{eq: distance between i and i+1 near centers}
      C_4^{-1}d(\Phi_{i+1}(g_1),\Phi_{i+1}(g_2)) \leq  d(\Phi_i(g_1),\Phi_i(g_2))
      \leq C_4 \sup_{|w-p_{i+1}|\leq 2} d(\Phi_{i+1}(g_1)(w),\Phi_{i+1}(g_2)(w)).
  \end{equation}
  For $|z-p_{i+1}|\geq 1$
  \begin{equation} \label{eq: distance between i and i+1 far from centers}
    \begin{split}
    d(\Phi_{i+1}(g_1),\Phi_{i+1}(g_2)) \leq \left(1+\frac{C_4}{|z-p_{i+1}|^3}\right)d(\Phi_i(g_1),\Phi_i(g_2)), \\
    d(\Phi_i(g_1),\Phi_i(g_2)) \leq \left(1+\frac{C_4}{|z-p_{i+1}|^3}\right)d(\Phi_{i+1}(g_1),\Phi_{i+1}(g_2)).
    \end{split}
  \end{equation}  
\end{itemize}

We set $\Phi := \Phi_n:X\to \mathcal{H}(\affine P^N)$ and 
$R_1 := 10R/\delta_3$.
We will check that they satisfy the required properties.
Take $g\in X$. For any $z\in \affine$, by (\ref{eq: error far from centers 2})
\[ |d\Phi(g)|(z)\leq \left(1+\sum_{p_i: |z-p_i|>R/2}\frac{C_4}{|z-p_i|^4}\right)
    \prod_{p_i: |z-p_i|>R/2}\left(1+\frac{C_4}{|z-p_i|^3}\right) <\lambda .\]
Here we have used (\ref{eq: difference between spherical derivatives under induction}) and
$\norm{d\Phi_i(g)}_{L^\infty(D_{R/2}(p_i))}<1$ for bad $p_i$.
Therefore the image of $\Phi$ is contained in $\moduli_\lambda(\affine P^N)$.

For each $p_i$ we can find $z\in D_R(p_i)$ satisfying 
$|df|(z)\geq \delta_3/(4R\sqrt{\pi})$ (if $p_i$ is good) or $|d\Phi_i(f)|(z)>1/100$ (if $p_i$ is bad).
Then by (\ref{eq: error far from centers 3}) 
\begin{equation*}
   \begin{split}
   |d\Phi(f)|(z) \geq 
   &\min\left(\frac{\delta_3}{4R\sqrt{\pi}},\frac{1}{100}\right)\prod_{p_j:|p_j-z|>R}\left(1-\frac{C_4}{|z-p_j|^3}\right) \\
   &-\left(\sum_{p_j: |z-p_j|>R}\frac{C_4}{|z-p_j|^4}\right)
    \prod_{p_j: |z-p_j|>R}\left(1+\frac{C_4}{|z-p_j|^3}\right) > \frac{\delta_3}{10R}.
   \end{split}
\end{equation*}
This shows that $\Phi(f)$ is $R_1$-nondegenerate over $\Lambda$ because $R_1=10R/\delta_3>10R$
and $\Lambda$ is contained in the union of $D_{2R}(p_i)$.

Let $g_1,g_2\in X$.
For any $z\in \affine$, by (\ref{eq: distance between i and i+1 near centers}) and 
(\ref{eq: distance between i and i+1 far from centers})
\begin{equation*}
   \begin{split}
   d(\Phi(g_1)(z),\Phi(g_2)(z)) &\leq C_4\prod_{p_i:|z-p_i|>1}\left(1+\frac{C_4}{|z-p_i|^3}\right)
   d(g_1(z),g_2(z))  \\
   &\leq 2C_4(1+C_4)d(g_1(z),g_2(z)), \quad (\text{by (\ref{eq: error far from centers 2})}),
   \end{split}
\end{equation*}
\begin{equation*}
   \begin{split}
   d(g_1(z),g_2(z)) &\leq \sup_{w:|w-z|\leq 3} 
   \left\{C_4\prod_{p_i:|w-p_i|>1}\left(1+\frac{C_4}{|w-p_i|^3}\right)d(\Phi(g_1)(w),\Phi(g_2)(w))\right\} \\
   &\leq 2C_4(1+C_4) \sup_{w:|w-z|\leq 3} d(\Phi(g_1)(w),\Phi(g_2)(w)) \quad 
   (\text{by (\ref{eq: error far from centers 2})}).
   \end{split}
\end{equation*}
We have finished the proof.
\end{proof}

\section{Quantitative study of nondegenerate curves} \label{section: quantitative study of nondegenerate curves}

In this section we study a neighborhood of a nondegenerate curve by a deformation technique and 
prove Proposition \ref{prop: local study around nondegenerate curve}.
Deformation theory itself is a well-established subject, but here we encounter two unorthodox issues:
\begin{itemize}
  \item We need thorough quantitative descriptions.
  \item The complex plane $\affine$ is non-compact. So we need to ``project'' the problem to a compact setting.
\end{itemize}
Subsection \ref{subsection: analysis of the operator (-Delta+1)} is an analytic preparation 
for the first issue.
The second issue is dealt with in Subsection \ref{subsection: constructing holomorphic sections over the torus}.
We prove Proposition \ref{prop: local study around nondegenerate curve} in 
Subsection \ref{subsection: proof of Proposition local study around nondegenerate curve}.

\subsection{Analysis of the operator $(-\Delta+1)$}  \label{subsection: analysis of the operator (-Delta+1)}

In this subsection we prepare some facts relating to the operator 
$-\Delta+1=-\partial^2/\partial x^2 - \partial^2/\partial y^2 + 1$.
The following lemma was proved in Matsuo--Tsukamoto \cite[Sublemma 4.4]{Matsuo--Tsukamoto Brody curves}
\begin{lemma} \label{lemma: L^infty estimate for functions}
Let $\varphi:\affine \to \mathbb{R}$ be a $C^2$ function 
such that $\varphi$ and $\Delta \varphi$ are bounded. Then 
\[ \norm{\varphi}_{L^\infty(\affine)} \leq 4\norm{(-\Delta+1)\varphi}_{L^\infty(\affine)}.\]
\end{lemma}

Let $\psi$ be a real valued function in $\affine$. 
We define the $C^k$ norm $\norm{\psi}_{C^k}$ as $\sum_{i=0}^k \norm{\nabla^i \psi}_{L^\infty(\affine)}$.
For $R>0$ and a subset $\Lambda\subset \affine$ the function $\psi$ is said to 
be $R$-nondegenerate over $\Lambda$ if 
\[ \forall a\in \Lambda: \quad \sup_{D_R(a)}\psi \geq 1/R^2.\]

\begin{lemma} \label{lemma: nondegeneracy implies positivity}
For any positive numbers $K$ and $R$ there exists $\kappa = \kappa(K,R)>0$ satisfying the following statement.
Let $\varphi$ be a real valued bounded $C^2$ function in $\affine$.
Suppose $\psi := (-\Delta+1)\varphi$ is a $C^1$ function and satisfies 
\[ \psi\geq -\kappa, \quad \norm{\psi}_{C^1}\leq K, \quad \text{$\psi$ is $R$-nondegenerate all over the plane}.\]
Then $\inf_{z\in \affine} \varphi(z)\geq \kappa$.
\end{lemma}

\begin{proof}
Suppose the statement is false.
Then for any $n\geq 1$ there exists a bounded $C^2$ function $\varphi_n$ in $\affine$ such that 
$\inf \varphi_n < 1/n$ and 
$\psi_n := (-\Delta+1)\varphi_n$ satisfies
\[ \psi_n \geq -1/n, \quad \norm{\psi_n}_{C^1}\leq K, \quad \text{$\psi_n$ is $R$-nondegenerate over $\affine$}.\]
By translation, we can assume $\varphi_n(0) < \inf \varphi_n + 1/n < 2/n$.
From Lemma \ref{lemma: L^infty estimate for functions}, 
the Schauder estimate \cite[Theorem 6.2]{Gilbarg--Trudinger} 
and $\norm{\psi_n}_{C^1}\leq K$, the functions $\varphi_n$ are bounded in 
$C^{2,\alpha}$ $(0<\alpha<1)$ over every compact subset of $\affine$.
By the Arzela--Ascoli theorem we can choose a sequence $n_1<n_2<n_3<\dots$ such that 
$\varphi_{n_k}$ and $\psi_{n_k}$ converge in $C^2$ and $C^0$ respectively 
over every compact subset of $\affine$.
We denote their limits by $\varphi$ and $\psi$.
They satisfy 
\[ \varphi(0) = \inf \varphi \leq 0, \quad 
   (-\Delta+1)\varphi=\psi, \quad \psi\geq 0, \quad \text{$\psi$ is $R$-nondegenerate over $\affine$}. \]
$\varphi$ achieves the non-positive minimum at the origin.
The condition $(-\Delta+1)\varphi\geq 0$ and the strong minimum principle \cite[Theorem 3.5]{Gilbarg--Trudinger} imply 
that $\varphi$ must be a non-positive constant.
But then $\psi=(-\Delta+1)\varphi=\varphi$ cannot be $R$-nondegenerate over $\affine$.
\end{proof}

Let $\Gamma\subset \affine$ be a lattice. We study the operator $-\Delta+1$ over the torus $\affine/\Gamma$.
It is very important that all the estimates in the lemmas below do not depend on $\Gamma$.
Indeed we can establish more general statements over the universal cover $\affine$.
But we don't need them.

\begin{lemma} \label{lemma: solving the Helmholtz equation}
Let $K$ and $R$ be positive numbers.
Let $\psi$ be a real valued $C^1$ function over the torus $\affine/\Gamma$ satisfying 
\[ \psi \geq -\kappa(K,R), \quad \norm{\psi}_{C^1}\leq K, \quad \text{$\psi$ is $R$-nondegenerate over $\affine/\Gamma$}.\]
(The last condition means that the pull-back of $\psi$ to the plane is $R$-nondegenerate over $\affine$.)
Then there exists a $C^2$ function $\varphi$ over the torus satisfying 
\[ (-\Delta+1)\varphi = \psi,\quad \norm{\varphi}_{C^2}\lesssim_K 1, \quad \varphi \geq \kappa(K,R).\]
\end{lemma}
\begin{proof}
If we find a $C^2$ function $\varphi$ in $\affine/\Gamma$ satisfying 
$(-\Delta+1)\varphi=\psi$, then it satisfies 
$\norm{\varphi}_{C^2}\lesssim_K 1$ (by Lemma \ref{lemma: L^infty estimate for functions} and 
the Schauder estimate) and $\varphi\geq \kappa$ (by Lemma \ref{lemma: nondegeneracy implies positivity}).
So the remaining problem is only to solve the equation, and this can be done by a standard $L^2$ technique.
Consider a bounded linear map:
\[ L^2_1(\affine/\Gamma)\to \mathbb{R}, \quad \phi \mapsto (\psi,\phi)_{L^2}.\]
By the Riesz representation theorem, we can find $L^2_1$ function $\varphi$ satisfying 
$(\nabla\varphi, \nabla\phi)_{L^2} + (\varphi, \phi)_{L^2} = (\psi, \phi)_{L^2}$ for all 
$\phi\in L^2_1$. Then $\varphi$ is a distributional solution of $(-\Delta+1)\varphi=\psi$.
By the $L^p$ estimate \cite[Theorem 9.11]{Gilbarg--Trudinger}, 
$\varphi$ is in $L^p_3$ for any $1<p<\infty$.
By the Sobolev embedding \cite[Corollary 7.11]{Gilbarg--Trudinger}, $\varphi$ is in $C^2$.
\end{proof}

So far we have considered the operator $-\Delta+1$ acting on functions.
We will also need the case of a vector bundle coefficient.
Let $\mathcal{E}$ be a holomorphic vector bundle over the torus $\affine/\Gamma$
with a Hermitian metric.
We denote its canonical connection by $\nabla$ and define the curvature 
$\Theta :=[\nabla_{\partial/\partial z}, \nabla_{\partial/\partial \bar{z}}]$.
We assume a positivity of $\Theta$; there exists a positive number $\sigma$ satisfying 
\[ \forall u\in \mathcal{E}: \> \langle \Theta u, u\rangle \geq \sigma |u|^2.\]
Let $\bar{\partial}:\Omega^0(\mathcal{E})\to \Omega^{0,1}(\mathcal{E})$ be the Dolbeault operator, 
and $\bar{\partial}^*:\Omega^{0,1}(\mathcal{E})\to \Omega^0(\mathcal{E})$ its formal adjoint.
The operator $\bar{\partial}\bar{\partial}^*$ is an analogue of $-\Delta+1$ because of
the Bochner formula
\begin{equation} \label{eq: Bochner formula}
  \bar{\partial}\bar{\partial}^* a = \frac{1}{2}\nabla^* \nabla a + \Theta a, \quad (a\in \Omega^{0,1}(\mathcal{E})).
\end{equation}
The next lemma is an analogue of Lemma \ref{lemma: L^infty estimate for functions} and proved 
in \cite[Proposition 4.2]{Tsukamoto deformation}.

\begin{lemma} \label{lemma: L^infty estimate}
Let $a\in \Omega^{0,1}(\mathcal{E})$ be in $C^2$.
Then 
\[ \norm{a}_{L^\infty(\affine/\Gamma)} \leq \frac{8}{\sigma} \norm{\bar{\partial}\bar{\partial}^* a}_{L^\infty(\affine/\Gamma)}.\]
\end{lemma}

\begin{lemma} \label{lemma: solving d-bar equation}
Suppose a positive number $K$ satisfies 
\[ \norm{\Theta}_{C^1_\nabla} := \norm{\Theta}_{L^\infty(\affine/\Gamma)} + \norm{\nabla\Theta}_{L^\infty(\affine/\Gamma)}
   \leq K.\]
Then for any $b\in L^2_2(\Omega^{0,1}(\mathcal{E}))$ there exists $a\in L^2_4(\Omega^{0,1}(\mathcal{E}))$
satisfying 
\[ \bar{\partial}\bar{\partial}^* a = b, \quad 
   \norm{a}_{L^\infty(\affine/\Gamma)} + \norm{\nabla a}_{L^\infty(\affine/\Gamma)} 
   \lesssim_{\sigma,K} \norm{b}_{L^\infty(\affine/\Gamma)}.\]
\end{lemma}
\begin{proof}
By the Bochner formula and the positive curvature condition, we can find $a\in L^2_4(\Omega^{0,1}(\mathcal{E}))$ satisfying 
$\bar{\partial}\bar{\partial}^* a=b$. 
This is the same as the proof of Lemma \ref{lemma: solving the Helmholtz equation}.
The problem is to show the required estimate.
We have $\norm{a}_{L^\infty}\lesssim_\sigma \norm{b}_{L^\infty}$ by Lemma \ref{lemma: L^infty estimate}.
We want to estimate $\norm{\nabla a}_{L^\infty}$.
By considering a finite covering, we can assume that the injectivity radius of the torus $\affine/\Gamma$ is greater than 
$3$.
Take a point $p\in \affine/\Gamma$.
From the condition $\norm{\Theta}_{C^1_\nabla}\leq K$, we can find a trivialization $g$ of $\mathcal{E}$ (as a Hermitian 
$C^\infty$ vector bundle) over the disk $D_2(p)$ such that the connection matrix $A$ representing $\nabla$ under $g$ satisfies 
\[ \norm{A}_{C^1(D_2(p))} \lesssim_K 1.\]
Such $g$ can be constructed by using the parallel translation along the geodesics from $p$.
Under the trivialization $g$, the operator $\bar{\partial}\bar{\partial}^*$ is expressed as 
\[  \bar{\partial}\bar{\partial}^* = \frac{-1}{2} \Delta + B_1\frac{\partial}{\partial x} + B_2\frac{\partial}{\partial y}
    + B_3,\]
where the $L^\infty$ norms of the coefficients $B_1, B_2,B_3$ over $D_2(p)$ are $\lesssim_K 1$.
Then by using the Sobolev embedding and the $L^3$ estimate 
\begin{equation*}
   \begin{split}
   \norm{\nabla a}_{L^\infty(D_1(p))} &\lesssim_K \norm{a}_{L^3_2(D_1(p))} \quad (\text{Sobolev embedding})\\
   &\lesssim_K \norm{a}_{L^3(D_2(p))} + \norm{b}_{L^3(D_2(p))} \quad (\text{$L^3$ estimate})\\   
   &\lesssim_{\sigma} \norm{b}_{L^\infty(\affine/\Gamma)}  \quad (\text{Lemma \ref{lemma: L^infty estimate}}).
   \end{split}
\end{equation*}
This holds for every point $p$. 
So we get $\norm{\nabla a}_{L^\infty(\affine/\Gamma)} \lesssim_{K,\sigma} \norm{b}_{L^\infty(\affine/\Gamma)}$.
\end{proof}

\subsection{Constructing holomorphic sections over the torus} \label{subsection: constructing holomorphic sections over the torus}

Let $E$ be a holomorphic vector bundle over the plane $\affine$
with a Hermitian metric $h$.
Let $\nabla$ and $\Theta$ be the canonical connection and
curvature of $(E,h)$ respectively.
We denote by $N$ the rank of $E$ and regard it as a universal constant.

\begin{lemma}  \label{lemma: sampling}
Suppose a positive number $K$ satisfies $\norm{\Theta}_{C^1_\nabla}\leq K$.
For any positive number $\varepsilon$ there exists a natural number $m=m(K,\varepsilon)$ satisfying the following 
statement.
Let $\square$ be a unit square in the plane. 
We can find $m$ points $p_1,\dots,p_m$ in $\square$ so that 
every $C^1$ section $u$ of $E$ over $D_1(\square)$ satisfies 
\[  \norm{u}_{L^\infty(\square)} \leq \max_{1\leq i\leq m} |u(p_i)| + 
    \varepsilon \left(\norm{u}_{L^\infty(D_1(\square))} + \norm{\bar{\partial}u}_{L^\infty(D_1(\square))}\right).\]
\end{lemma}
\begin{proof}
Take a small positive number $\delta = \delta(K,\varepsilon)$. 
Let $\{p_1,\dots,p_m\}\subset \square$, $m=m(\delta)$, be a $\delta$-dense subset 
(i.e. $\square$ is contained in the union of $D_\delta(p_i)$).
From the Sobolev embedding 
$L^3_1(\mathbb{R}^2)\hookrightarrow C^{0,1/3}(\mathbb{R}^2)$ (\cite[Theorem 7.17]{Gilbarg--Trudinger}),
\[ \norm{u}_{L^\infty(\square)} \leq \max_{1\leq i\leq m}|u(p_i)| 
   + \const\cdot \delta^{1/3}\norm{\nabla u}_{L^3(D_{1/2}(\square))}.\]
As in the proof of Lemma \ref{lemma: solving d-bar equation}, the $L^3$ estimate implies
\[ \norm{\nabla u}_{L^3(D_{1/2}(\square))} \lesssim_K 
   \norm{u}_{L^\infty(D_1(\square))} + \norm{\bar{\partial}u}_{L^\infty(D_1(\square))}.\]
We choose $\delta \ll \tau^3$. Then we get the statement.
\end{proof}

Let $L$ be a positive number and set $\Lambda := [0,L]^2$. 
We consider the following condition.
\begin{condition}\label{condition: nondegeneracy of vector bundle}
There exist positive numbers $K, R$ and a non-negative $C^1$ function $\phi$ in the plane such that 
\begin{itemize}
  \item $\norm{\Theta}_{C^1_\nabla}\leq K$ and $\norm{\phi}_{C^1}\leq K$.
  \item $\langle \Theta u, u\rangle \geq \phi(z) |u|^2$ for all $z\in \affine$ and $u\in E_z$.
  \item $\phi$ is $R$-nondegenerate over the square $\Lambda$.
\end{itemize}
\end{condition}

Let $\Gamma :=(L+8)(\mathbb{Z} + \mathbb{Z}\sqrt{-1})$ be the lattice in $\affine$ generated by $L+8$ and $(L+8)\sqrt{-1}$.
We consider the torus $\affine/\Gamma$.

\begin{lemma} \label{lemma: constructing holomorphic sections over the torus}
There exists a positive number $L_1=L_1(K,R)$ such that 
if $L\geq L_1$ and Condition \ref{condition: nondegeneracy of vector bundle} holds
then we can construct a Hermitian holomorphic vector bundle $\mathcal{E}$ over the torus $\affine/\Gamma$
and a linear map $T:C^\infty((-1,L+1)^2,E)\to H^0(\affine/\Gamma, \mathcal{E})$ (a map from the 
space of $C^\infty$ sections of $E$ over $(-1,L+1)^2$ to the space of holomorphic sections 
of $\mathcal{E}$) satisfying the following conditions.

\noindent 
(1) 
\[ \left| \dim_\affine H^0(\affine/\Gamma, \mathcal{E})
   -\frac{1}{\pi}\int_\Lambda \mathrm{tr}(\Theta) dxdy\right| \lesssim_K L.\]

\noindent 
(2) For any $u\in C^\infty((-1,L+1)^2,E)$
\[ \norm{Tu}_{L^\infty(\affine/\Gamma)} \lesssim_{K,R} \norm{u}_{L^\infty((-1,L+1)^2)} + 
   \norm{\bar{\partial}u}_{L^\infty((-1,L+1)^2)},\]  
\[ \norm{u}_{L^\infty(\Lambda)} \lesssim_{K,R} \norm{Tu}_{L^\infty(\affine/\Gamma)} 
   + \norm{u}_{L^\infty((-1,L+1)^2\setminus \Lambda)} +\norm{\bar{\partial}u}_{L^\infty((-1,L+1)^2)}.\]
\end{lemma}

\begin{proof}
Let $\alpha$ be a smooth non-decreasing function in $\mathbb{R}$ such that 
$\alpha(t)=t$ over $-1\leq t\leq L+1$, $\alpha'(t)=0$ over $\{t\leq -2\}\cup\{t\geq L+2\}$ and 
$|\alpha^{(k)}|\lesssim_k 1$ for $k\geq 1$.
We define $\Phi:\affine\to \affine$ by $\Phi(x,y):=(\alpha(x),\alpha(y))$.
Let $(E',h')$ be the pull-back of $(E,h)$ by the map $\Phi$.
This is a $C^\infty$ Hermitian vector bundle over the plane.
By Condition \ref{condition: nondegeneracy of vector bundle}, the curvature of the pull-back connection $\Phi^*\nabla$ satisfies 
\[ \langle \Theta_{\Phi^*\nabla}u, u\rangle \geq \alpha'(x)\alpha'(y) \phi(\alpha(x),\alpha(y))|u|^2, \quad 
   (u\in E'_{(x,y)}).\]
$\Phi^*\nabla$ is flat outside of $[-2,L+2]^2$.
We can classify flat unitary connections over the complement of $[-2,L+2]^2$ by their holonomy maps 
$\pi_1(\affine\setminus [-2,L+2]^2) \to U(N)$.
Then there exists a trivialization $g$ of $(E',h')$ outside of $[-2,L+2]^2$ such that the connection matrix 
$A$ of $\Phi^*\nabla$ with respect to $g$ satisfies $|A|\lesssim 1/L$.
(More precisely we can assume that $A$ has the form $A=a dx$ with $|a|\lesssim 1/L$.)
We take a cut-off $\beta:\affine\to [0,1]$ such that 
$\beta=1$ on $[-2,L+2]^2$ and $\beta=0$ outside of $[-3,L+3]^2$.
We define a unitary connection $\nabla'$ on $E'$ by 
$\nabla' := g^{-1}(d+\beta A)$.

We define a $C^\infty$ complex vector bundle $\mathcal{E}$ on $\affine/\Gamma$ as follows.
We naturally regard the torus $\affine/\Gamma$ as the quotient of the square 
$[-4,L+4]^2$ (the opposite edges are identified).
We glue the bundle $E'|_{(-4,L+4)^2}$ and $\left([-4,L+4]^2\setminus [-3,L+3]^2\right)\times \affine^N$
by the trivialization $g$.
This provides a $C^\infty$ vector bundle on the torus $\affine/\Gamma$.
We denote it by $\mathcal{E}$. 
The metric $h'$ and the connection $\nabla'$ descend to the torus and become a metric and a unitary connection on $\mathcal{E}$.
Let $\bar{\partial}_{\mathcal{E}}:\Omega^0(\mathcal{E})\to \Omega^{0,1}(\mathcal{E})$ be the $(0,1)$-part of the 
covariant derivative $\nabla': \Omega^0(\mathcal{E})\to \Omega^1(\mathcal{E})$.
This provides a holomorphic structure on $\mathcal{E}$.
The curvature of the connection $\nabla'$ satisfies 
\[ \langle \Theta_{\nabla'}u, u\rangle \geq -\frac{c}{L}|u|^2 \quad (u\in \mathcal{E})\]
outside of $[-2,L+2]^2$.
Here $c$ is a universal constant.
We define a function $\psi$ in the torus $\affine/\Gamma$ by 
\[ \psi(x,y) := \alpha'(x)\alpha'(y) \phi(\alpha(x),\alpha(y)) - \frac{c}{L}\quad (-4\leq x,y\leq L+4).\]
This satisfies $\langle \Theta_{\nabla'}u, u\rangle \geq \psi |u|^2$ all over the torus, and 
\[ \psi\geq -\frac{c}{L}, \quad \norm{\psi}_{C^1}\lesssim_K 1.\]
Since $\phi$ is $R$-nondegenerate over $[0,L]^2$, if $L$ is sufficiently larger than $R^2$ then 
$\psi$ is $(R+10)$-nondegenerate over the torus $\affine/\Gamma$.
By Lemma \ref{lemma: solving the Helmholtz equation} if $L$ is 
sufficiently large then we can find $\varphi:\affine/\Gamma\to \mathbb{R}$ satisfying 
\[ (-\Delta+1)\varphi=4\psi, \quad \norm{\varphi}_{C^2}\lesssim_{K} 1, \quad 
   \varphi \gtrsim_{K,R} 1. \]
We define a metric on $\mathcal{E}$ by $h_{\mathcal{E}} := e^{-\varphi}h'$.
Let $\nabla_{\mathcal{E}}$ and $\Theta_{\mathcal{E}}$ be the canonical connection and  
curvature of the Hermitian holomorphic vector bundle $(\mathcal{E},h_{\mathcal{E}})$.
We have $\Theta_{\mathcal{E}} = \Theta_{\nabla'} +\Delta\varphi/4 = \Theta_{\nabla'} -\psi +\varphi/4$.
This enjoys a positivity:
\[ \langle \Theta_{\mathcal{E}}u, u\rangle \geq (\varphi/4)|u|^2 \gtrsim_{K,R} |u|^2 \quad (u\in \mathcal{E}).\]
It also satisfies $\norm{\Theta_{\mathcal{E}}}_{C^1_{\nabla_{\mathcal{E}}}}\lesssim_{K} 1$.
Therefore $(\mathcal{E},h_{\mathcal{E}})$ satisfies the assumption of Lemma \ref{lemma: solving d-bar equation}.

So far we have constructed the Hermitian holomorphic vector bundle $(\mathcal{E},h_{\mathcal{E}})$.
Next we will construct a linear map $T:C^\infty((-1,L+1)^2,E)\to H^0(\affine/\Gamma, \mathcal{E})$.
Take a cut-off $\gamma:\affine \to [0,1]$ such that 
$\gamma=1$ over $\Lambda = [0,L]^2$ and it is supported in $(-1,L+1)^2$.
Note that $(\mathcal{E},h_{\mathcal{E}})$ is naturally identified with $(E,h)$ over $(-1,L+1)^2$.
Let $u\in C^\infty((-1,L+1)^2,E)$ and consider $\gamma u$. 
This is supported in $(-1,L+1)^2$, and then it is naturally identified with a section of $\mathcal{E}$ over the torus.
We define a holomorphic section $T(u)\in H^0(\affine/\Gamma, \mathcal{E})$ by 
\begin{equation} \label{eq: definition of a linear map T}
 T(u) := \gamma u - 
 \bar{\partial}_{\mathcal{E}}^*(\bar{\partial}_{\mathcal{E}}\bar{\partial}_{\mathcal{E}}^*)^{-1} 
 \left(\bar{\partial}_{\mathcal{E}}(\gamma u)\right).
\end{equation}
Here the inverse $(\bar{\partial}_{\mathcal{E}}\bar{\partial}_{\mathcal{E}}^*)^{-1}$ is provided by Lemma 
\ref{lemma: solving d-bar equation}.

We check the conditions (1) and (2).
First, by the positivity of the curvature $\Theta_{\mathcal{E}}$, the cohomology $H^1(\affine/\Gamma, \mathcal{E})$
vanishes. Then by the Riemann--Roch formula
\[ \dim_\affine H^0(\affine/\Gamma,\mathcal{E}) = \int_{\affine/\Gamma}c_1(\mathcal{E}) 
  = \frac{1}{\pi}\int_{\affine/\Gamma} \mathrm{tr}(\Theta_{\mathcal{E}})dxdy. \]
We have $\Theta_{\mathcal{E}}=\Theta$ over $[-1,L+1]^2$ and $|\Theta_{\mathcal{E}}|\lesssim_K 1$. Hence
\[   \dim_\affine H^0(\affine/\Gamma,\mathcal{E}) = \frac{1}{\pi}\int_\Lambda \mathrm{tr}(\Theta)dxdy
     + O(L),\]
where the implicit constant depends only on $K$. This shows the condition (1).     
  
Second, we look at the formula (\ref{eq: definition of a linear map T}) and use Lemma \ref{lemma: solving d-bar equation}:
\[ \norm{Tu}_{L^\infty(\affine/\Gamma)} \lesssim_{K,R} \norm{u}_{L^\infty((-1,L+1)^2)}
    + \norm{\bar{\partial}_{\mathcal{E}}(\gamma u)}_{L^\infty(\affine/\Gamma)}
    \lesssim \norm{u}_{L^\infty((-1,L+1)^2)} + \norm{\bar{\partial}u}_{L^\infty((-1,L+1)^2)}.\]
This is the first half of the condition (2). The other also follows from Lemma \ref{lemma: solving d-bar equation}.
\begin{equation*}
   \begin{split}
   \norm{u}_{L^\infty(\Lambda)} &\leq \norm{\gamma u}_{L^\infty((-1,L+1)^2)} 
   \leq \norm{Tu}_{L^\infty(\affine/\Gamma)} + 
   \norm{\bar{\partial}_{\mathcal{E}}^*(\bar{\partial}_{\mathcal{E}}\bar{\partial}_{\mathcal{E}}^*)^{-1} 
   \left(\bar{\partial}_{\mathcal{E}}(\gamma u)\right)}_{L^\infty(\affine/\Gamma)} \\
   &\lesssim_{K,R} \norm{Tu}_{L^\infty(\affine/\Gamma)} + \norm{\bar{\partial}_{\mathcal{E}}(\gamma u)}_{L^\infty(\affine/\Gamma)}
   \quad (\text{Lemma \ref{lemma: solving d-bar equation}})\\
   &\lesssim \norm{Tu}_{L^\infty(\affine/\Gamma)} + \norm{u}_{L^\infty((-1,L+1)^2\setminus \Lambda)}
     + \norm{\bar{\partial} u}_{L^\infty ((-1,L+1)^2)}.
   \end{split}
\end{equation*}
This finishes the proof.
\end{proof}

\subsection{Proof of Proposition \ref{prop: local study around nondegenerate curve}} 
\label{subsection: proof of Proposition local study around nondegenerate curve}

Let $T\affine P^N$ be the tangent bundle of the projective space, and let 
\[ \exp:T\affine P^N\to \affine P^N \]
be the exponential map with respect to the Fubini--Study metric.

\begin{lemma} \label{lemma: nonlinear CR equation}
For any $\varepsilon>0$ there exists $\delta>0$ satisfying the following statement.
Let $p\in \affine$ and $f,g_1,g_2\in \moduli_2(\affine P^N)$, and suppose 
\[   \sup_{z\in D_1(p)} d(f(z),g_i(z))\leq \delta \quad (i=1,2).\]
Take $u_i\in C^\infty(D_1(p), f^*T\affine P^N)$ with $g_i(z) = \exp_{f(z)}u_i(z)$ and $|u_i|\leq \delta$. 
Then 
\[ |\bar{\partial}u_1(p)-\bar{\partial}u_2(p)| \leq \varepsilon \norm{u_1-u_2}_{L^\infty(D_1(p))}.\]
\end{lemma}
\begin{proof}
We can assume $p=0$ and $f(p)=[1:0:\dots:0]$.
Let $[1:w_1:\dots:w_N]$ be the standard coordinate around $f(p)$.
Set $w:=(w_1,\dots,w_N)$.
We choose $\delta>0$ so small that $f(D_{1/10})$, $g_1(D_{1/10})$ and $g_2(D_{1/10})$
are all contained in $\{[1:w]|\, |w|<1\}$.
For $|w|<2$ and $v=(v_1,\dots,v_N)\in \affine^N$ with $|v|\ll 1$, we set
\[ [1:\zeta] := \exp_{[1:w]}\left(\sum_{n=1}^N v_n\frac{\partial}{\partial w_n}\right).\]
$\zeta$ can be expressed as 
$\zeta = w+v+P(w,v)$ with a $C^\infty$ function $P(w,v)$ satisfying 
$P(w,0)=0$ and $\nabla_v P(w,0)=0$.
By the implicit function theorem we can write $v$ as a function of $w$ and $\zeta$:
\[ v=\zeta-w + Q(w,\zeta-w),\]
where $Q(w,\xi)$ is a $C^\infty$ function satisfying 
\begin{equation} \label{eq: conditions on Q}
   Q(w,0)=0, \quad  \nabla_\xi Q(w,0)=0.
\end{equation}

Let $f(z)=[1:F(z)]$, $g_i(z)=[1:G_i(z)]$ and $u_i(z) = \sum_{n=1}^N v_{in}\partial/\partial w_n$
over $D_{1/10}$.
Set $v_i := (v_{i1},\dots,v_{iN})$. We have $|v_i|\leq \delta \ll 1$ and
\[ v_i(z) = G_i(z)-F(z) + Q(F(z),G_i(z)-F(z)). \]
Differentiating this, 
\[ \frac{\partial v_i}{\partial \bar{z}} = \nabla_w Q(F(z),G_i(z)-F(z))* F'(z)
    +\nabla_{\xi} Q(F(z),G_i(z)-F(z))*(G_i'(z) - F'(z)).\]
We have $|G_i(p)-F(p)|\lesssim \delta$ and $|G'_i(p)-F'(p)|\lesssim \delta$ by the Cauchy estimate.
By (\ref{eq: conditions on Q})
\begin{equation*}
   \begin{split}
   \left|\frac{\partial v_1}{\partial \bar{z}}(p)-\frac{\partial v_2}{\partial \bar{z}}(p)\right|
   &\lesssim \delta |G_1(p)-G_2(p)| + \delta |G'_1(p)-G'_2(p)| \\
   &\lesssim \delta \sup_{z\in D_{1/10}}|G_1(z)-G_2(z)| \quad (\text{Cauchy estimate})\\
   &\lesssim \delta \norm{u_1-u_2}_{L^\infty(D_{1/10})}.
   \end{split}
\end{equation*}
\end{proof}

We restate Proposition \ref{prop: local study around nondegenerate curve}.

\begin{proposition}[= Proposition \ref{prop: local study around nondegenerate curve}]
For any $R>0$ and $0<\varepsilon<1$ there exist positive numbers $\delta_2=\delta_2(R)$, 
$C_2=C_2(R)$ and $C_3=C_3(\varepsilon)$ satisfying the following statement.
Let $f\in \moduli_2(\affine P^N)$, and let $\Lambda\subset \affine$ be a square of side length $L\geq 1$.
Suppose $f$ is $R$-nondegenerate over $\Lambda$. Then 
\begin{equation*}
  \#_{\mathrm{sep}}\left(\{g\in \moduli_2(\affine P^N)|\, \mathbf{d}_{D_5(\Lambda)}(f,g)\leq \delta_2\}, 
  \mathbf{d}_{\Lambda},\varepsilon\right) 
  \leq (C_2/\varepsilon)^{2(N+1)\int_{\Lambda}|df|^2 dxdy + C_3 L}.
\end{equation*}
\end{proposition}

\begin{proof}
We can assume $\Lambda=[0,L]^2$ without loss of generality.
We define a Hermitian holomorphic vector bundle $(E,h)$ over the plane as the pull-back of the tangent bundle $T\affine P^N$ 
and the Fubini--Study metric by the map $f$.
Let $\nabla$ and $\Theta =[\nabla_{\partial/\partial z},\nabla_{\partial/\partial \bar{z}}]$ be its
canonical connection and curvature. 
The Fubini--Study metric is a K\"{a}hler--Einstein metric, and its holomorphic bisectional curvature
is bounded from below by $2\pi$.
Then we have 
\begin{equation*} 
  \mathrm{tr}(\Theta)=\pi(N+1)|df|^2, \quad \langle \Theta u, u\rangle \geq \pi |df|^2 |u|^2 \quad (u\in E).
\end{equation*}
The function $\phi:= \pi |df|^2$ is $R$-nondegenerate over $\Lambda$ because of the nondegeneracy of $f$.
As $\moduli_2(\affine P^N)$ is compact, there exists a universal positive constant $K$ satisfying 
\[ \norm{\Theta}_{C^1_\nabla} \leq K, \quad  \norm{\phi}_{C^1} \leq K.\]
Therefore $(E,h)$ satisfies Condition \ref{condition: nondegeneracy of vector bundle}.
Let $L_1=L_1(K,R)$ be the positive constant introduced in Lemma \ref{lemma: constructing holomorphic sections over the torus}.
If $1\leq L\leq L_1$, then the statement of the proposition is obvious because $C_3(\varepsilon)$ can be chosen 
arbitrarily large. 
So we can assume $L\geq L_1$ and use Lemma \ref{lemma: constructing holomorphic sections over the torus}.
It provides us a Hermitian holomorphic vector bundle $\mathcal{E}$ over the torus 
$\affine/\Gamma = \affine/(L+8)(\mathbb{Z}+\sqrt{-1}\mathbb{Z})$ and a linear map 
$T:C^\infty((-1,L+1)^2,E)\to H^0(\affine/\Gamma,\mathcal{E})$ satisfying the following conditions.
(Note $\mathrm{tr}(\Theta)=\pi(N+1)|df|^2$.)
\begin{itemize}
  \item $\dim_\affine H^0(\affine/\Gamma,\mathcal{E}) = (N+1)\int_\Lambda |df|^2 dxdy + O(L)$
  where the implicit constant in $O(L)$ is universal.
  \item For any $u\in C^\infty((-1,L+1)^2,E)$
  \begin{equation} \label{eq: bound on Tu by u}
   \norm{Tu}_{L^\infty(\affine/\Gamma)} \lesssim_R \norm{u}_{L^\infty((-1,L+1)^2)} 
   + \norm{\bar{\partial}u}_{L^\infty((-1,L+1)^2)}, 
  \end{equation}
  \begin{equation} \label{eq: bound on u by Tu}
    \norm{u}_{L^\infty(\Lambda)} \leq C\left(\norm{Tu}_{L^\infty(\affine/\Gamma)}+\norm{u}_{L^\infty((-1,L+1)^2\setminus \Lambda)}
    +\norm{\bar{\partial}u}_{L^\infty((-1,L+1)^2)}\right),
  \end{equation}
  where $C=C(R)>1$ is a constant depending only on $R$.
\end{itemize}

By Lemma \ref{lemma: sampling}
we can choose points $p_1,\dots,p_M\in [-2,L+2]^2\setminus \Lambda$ with $M\lesssim_{\varepsilon} L$ such that
every $C^1$ section $u$ of $E$ over $D_4(\Lambda)$ satisfies 
\begin{equation}\label{eq: sampling estimate}
  \norm{u}_{L^\infty((-2,L+2)^2\setminus \Lambda)} \leq \max_{1\leq n\leq M}|u(p_n)|
  + \varepsilon \left(\norm{u}_{L^\infty(\partial_4 \Lambda)} + \norm{\bar{\partial}u}_{L^\infty(\partial_4\Lambda)}\right).
\end{equation}
Here recall that $\partial_4\Lambda$ is  
the set of $z\in \affine$ satisfying $D_4(z)\cap \partial\Lambda\neq \emptyset$.

Take a small positive number $\delta_2$, 
and let $X$ be the set of $g\in \moduli_2(\affine P^N)$
satisfying $\mathbf{d}_{D_5(\Lambda)}(f(z),g(z))\leq \delta_2$.
We define a map $S$ from $X$ to $H^0(\affine/\Gamma, \mathcal{E})\oplus \bigoplus_{n=1}^M E_{p_n}$
as follows.
Let $g\in X$ and write it as $g(z)=\exp_{f(z)}u(z)$ with $u\in C^\infty(D_5(\Lambda),E)$ and
$|u|\leq \delta_2$. 
We set 
\[ S(g) := (T(u), u(p_1),\dots,u(p_M)).\]
(Strictly speaking, $T(u)$ means $T(u|_{(-1,L+1)^2})$.)
We define a norm $\nnorm{\cdot}$ on $H^0(\affine/\Gamma, \mathcal{E})\oplus \bigoplus_{n=1}^M E_{p_n}$
by $\nnorm{(v,v_1,\dots,v_M)} := \norm{v}_{L^\infty(\affine/\Gamma)}+\max_{1\leq n\leq M}|v_n|$.
From (\ref{eq: bound on Tu by u}) and Lemma \ref{lemma: nonlinear CR equation}
we have $\nnorm{S(g)}\lesssim_R \delta_2$.
So we can assume that the image of $S$ is contained in the unit ball of $H^0(\affine/\Gamma,\mathcal{E})\oplus 
\bigoplus_{n=1}^M E_{p_n}$.

Take $g_1$ and $g_2$ in $X$, and let $g_i=\exp_f u_i$ with $u_i\in C^\infty(D_5(\Lambda),E)$ and
$|u_i|\leq \delta_2$.
By the estimate (\ref{eq: bound on u by Tu}) 
\begin{equation*}
   \begin{split}
   \norm{u_1-u_2}_{L^\infty(\Lambda)} \leq C\Bigl(\norm{Tu_1-Tu_2}_{L^\infty(\affine/\Gamma)}
   + \norm{u_1-u_2}_{L^\infty((-1,L+1)^2\setminus \Lambda)} \\
   + \norm{\bar{\partial}u_1-\bar{\partial}u_2}_{L^\infty((-1,L+1)^2)}\Bigr).
   \end{split}
\end{equation*}
By Lemma \ref{lemma: nonlinear CR equation}, we can choose $\delta_2=\delta_2(R)$ so small that 
\[ \norm{\bar{\partial}u_1-\bar{\partial}u_2}_{L^\infty((-1,L+1)^2)}
   \leq \frac{1}{2C}\norm{u_1-u_2}_{L^\infty((-2,L+2)^2)}.\]
Then 
\[ \norm{u_1-u_2}_{L^\infty(\Lambda)} \leq 2C\norm{Tu_1-Tu_2}_{L^\infty(\affine/\Gamma)}
   + (2C+1)\norm{u_1-u_2}_{L^\infty((-2,L+2)^2\setminus \Lambda)}.\]
Applying (\ref{eq: sampling estimate}) to $u=u_1-u_2$, the term $\norm{u_1-u_2}_{L^\infty((-2,L+2)^2\setminus \Lambda)}$
is bounded by
\[  \max_{1\leq n\leq M}|u_1(p_n)-u_2(p_n)|
   + \varepsilon \left(\norm{u_1-u_2}_{L^\infty(\partial_4\Lambda)}
   +\norm{\bar{\partial}u_1-\bar{\partial}u_2}_{L^\infty(\partial_4\Lambda)}\right). \]
By $|u_i|\leq \delta_2$ in $D_5(\Lambda)$ and Lemma \ref{lemma: nonlinear CR equation},
$\norm{u_1-u_2}_{L^\infty(\partial_4\Lambda)}
   +\norm{\bar{\partial}u_1-\bar{\partial}u_2}_{L^\infty(\partial_4\Lambda)}\lesssim \delta_2$.
Thus 
\[ \norm{u_1-u_2}_{L^\infty(\Lambda)}\leq (2C+1)\nnorm{S(g_1)-S(g_2)} + \const_R\cdot \delta_2 \varepsilon.\]
As $d(g_1(z),g_2(z))\lesssim |u_1(z)-u_2(z)|$, we can choose $\delta_2\ll 1$ so that 
\[ \mathbf{d}_\Lambda(g_1,g_2) \leq C'\nnorm{S(g_1)-S(g_2)} + \frac{\varepsilon}{2}.\]
Here $C'=C'(R)>1$ depends only on $R$.
Then by Lemma \ref{lemma: separated set}
\[ \#_{\mathrm{sep}}(X,\mathbf{d}_\Lambda,\varepsilon) 
   \leq \#_{\mathrm{sep}}\left(S(X),\nnorm{\cdot},\frac{\varepsilon}{2C'}\right).\]
$S(X)$ is contained in the unit ball of $H^0(\affine/\Gamma, \mathcal{E})\oplus\bigoplus_{n=1}^M E_{p_n}$. So this is bounded by   
\begin{equation*}
   \#_{\mathrm{sep}}
   \left(B_1\left(H^0(\affine/\Gamma,\mathcal{E})\oplus\bigoplus_{n=1}^M E_{p_n}\right),\nnorm{\cdot},\frac{\varepsilon}{2C'}\right)
   \leq  
   \left(\frac{6C'}{\varepsilon}\right)^{\dim_\mathbb{R}\left(H^0(\affine/\Gamma,\mathcal{E})\oplus\bigoplus_{n=1}^M E_{p_n}\right)}.
\end{equation*}   
Here we have used Example \ref{example: separated set of Banach ball}.
The real dimension of $H^0(\affine/\Gamma,\mathcal{E})\oplus\bigoplus_{n=1}^M E_{p_n}$ is equal to 
\[ 2\dim_\affine H^0(\affine/\Gamma,\mathcal{E}) + 2NM = 2(N+1)\int_\Lambda |df|^2 dxdy + O(L) +2NM. \]
Recall that $M\lesssim_\varepsilon L$ and the implicit constant in $O(L)$ is universal.
So this proves the proposition.
\end{proof}

\vspace{0.5cm}

\address{ Masaki Tsukamoto \endgraf
Department of Mathematics, Kyoto University, Kyoto 606-8502, Japan}

\textit{E-mail address}: \texttt{tukamoto@math.kyoto-u.ac.jp}

\end{document}